\documentclass[10pt, oneside, reqno]{amsart}

\usepackage{epsfig}
\usepackage{amsthm}
\usepackage{amssymb}
\usepackage{amsmath}
\usepackage{amscd}
\usepackage{color}
\usepackage{esint}
\usepackage{enumitem}
\usepackage{bbm}
\usepackage[top=1.1in, bottom=1in, left=1in, right=1in]{geometry}

\usepackage[usenames,dvipsnames]{xcolor}
\usepackage[colorlinks=true, pdfstartview=FitV, linkcolor=blue, citecolor=blue, urlcolor=blue]{hyperref}

\newcommand{\rf}[1]{(\ref{#1})}

\def \<{\langle}
\def \>{\rangle}

\newcommand{\bg}{\begin{equation}}
\newcommand{\ed}{\end{equation}}
\newcommand{\bga}{\begin{eqnarray}}
\newcommand{\eda}{\end{eqnarray}}

\def\cbdu{\par{\raggedleft$\Box$\par}}

\newtheorem {Theorem}  {Theorem}

\numberwithin{Theorem}{section}

\newtheorem {Lemma}[Theorem]  {Lemma}
\newtheorem {Proposition}[Theorem]{Proposition}
\theoremstyle{definition}
\newtheorem{Definition}[Theorem]{Definition}
\theoremstyle{remark}
\newtheorem{Remark}[Theorem]{\bf Remark}

\def \l{\lambda}
\expandafter\chardef\csname pre amssym.def
at\endcsname=\the\catcode`\@ \catcode`\@=11
\def\undefine#1{\let#1\undefined}
\def\newsymbol#1#2#3#4#5{\let\next@\relax
 \ifnum#2=\@ne\let\next@\msafam@\else
 \ifnum#2=\tw@\let\next@\msbfam@\fi\fi
 \mathchardef#1="#3\next@#4#5}
\def\mathhexbox@#1#2#3{\relax
 \ifmmode\mathpalette{}{\m@th\mathchar"#1#2#3}%
 \else\leavevmode\hbox{$\m@th\mathchar"#1#2#3$}\fi}
\def\hexnumber@#1{\ifcase#1 0\or 1\or 2\or 3\or 4\or 5\or 6\or 7\or 8\or
 9\or A\or B\or C\or D\or E\or F\fi}

\font\teneufm=eufm10 \font\seveneufm=eufm7 \font\fiveeufm=eufm5
\newfam\eufmfam
\textfont\eufmfam=\teneufm \scriptfont\eufmfam=\seveneufm
\scriptscriptfont\eufmfam=\fiveeufm

\catcode`\@=\csname pre amssym.def at\endcsname

\newcounter{remark}
\newenvironment{remark}
{\medskip \stepcounter{remark} \noindent \textit{Remark
\arabic{section}.\arabic{remark}.}}{\rm \cbdu}

\numberwithin{equation}{section}
\numberwithin{figure}{section}

\def \onethird {\frac{1}{3}}

\newcommand{\supp}{{\mathit supp}\,}

\newcommand{\divv}{{\text {div}}\,}

\newcommand{\e}{\epsilon}

\renewcommand{\l}{\lambda}

\renewcommand{\th}{\theta}
\renewcommand{\b}{\beta}

\newcommand{\R}{\mathbf{R}}

\newcommand{\les}{\lesssim}
\renewcommand{\div}{\mbox{div}}

\newcommand{\Bb}{{\mathcal B}}

\newcommand{\Ff}{{\mathcal F}}

\newcommand{\Nn}{{\mathcal N}}

\newcommand{\Ss}{{\mathcal S}}
\newcommand{\Tt}{{\mathcal T}}

\newcommand{\tl}{\Tilde{\lambda}}
\newcommand{\brho}{\Bar{\rho}}
\newcommand{\bv}{\Bar{v}}

\newcommand\powermone[1]{#1^{-1}}

\newcommand{\wtj}{\widetilde{J}}

\def  \R   {{\mathbb R}}
\def  \Z   {{\mathbb Z}}
\def  \N   {{\mathbb N}}

\def  \T   {{\mathbb T}}

\def  \haf  {{\frac{1}{2}}}
\def  \12  {{\frac{1}{2}}}
\def  \p   {\partial}

\def  \initial   {{\textit{in}}}
\def  \Dq    {\Delta_q}

\def  \sumN  {\sum_{n=N}^\infty}
\def  \td     {\tilde{d}}

\newcommand\twonorm[1]{\lVert#1\rVert_{L^2(\T^3)}}
\newcommand\threenorm[1]{\lVert#1\rVert_{L^3(\T^3)}}
\newcommand\rnorm[1]{\Vert#1\Vert_{L^r(\T^3)}}
\newcommand\Linfnorm[1]{\Vert#1\Vert_{L^\infty(\T^3)}}

\newcommand\twonormRtwo[1]{\lVert#1\rVert_{L^2(\R^2)}}

\def\build#1_#2^#3{\mathrel{\mathop{\kern 0pt#1}\limits_{#2}^{#3}}}

\begin{document}

\title[Anomalous dissipation]{Anomalous Dissipation at Onsager-Critical Regularity}

\author [Alexey Cheskidov]{Alexey Cheskidov}
\address{Institute for Theoretical Sciences, Westlake University.}
\email{cheskidov@westlake.edu.cn} 
\author [Qirui Peng]{Qirui Peng}
\address{Department of Mathematics, Department of Mathematics, University of California, Santa Barbara,CA 93117,USA}
\email{qpeng9@ucsb.edu}

\thanks{}





\maketitle

\begin{abstract}
We construct solutions to the three-dimensional Euler equations exhibiting anomalous dissipation in finite time through a vanishing viscosity limit. Inspired by \cite{BDL23} and \cite{cheskidov2023dissipation}, we extend the $2\frac{1}{2}$-dimensional constructions and establish an Onsager-critical energy criterion adapted to such flows, showing its sharpness. Moreover, we provide a fully three-dimensional dissipative Euler example, sharp in Onsager’s sense, driven by a slightly rough external force following the framework of \cite{CL21}.


CLASSIFICATION CODE: 35Q30; 76F02; 76D05; 35A02.

\end{abstract}

\section{Introduction}\label{Sec_Intro}

Turbulent fluid motion is well known for exhibiting the phenomenon of \textit{viscous dissipation anomaly}, according to which energy dissipation does not go to zero in the limit of vanishing viscosity. Back in 1941, Kolmogorov’s theory of turbulence postulated that in the limit as the Reynolds number $Re$ goes $\infty$, the dimensionless dissipation rate approaches a positive constant.  This prediction has been extensively confirmed by experiments and simulations of the Navier–Stokes equations at high Reynolds numbers. 

A few years later, Onsager proposed an alternative viewpoint.
In 1949 he conjectured that even solution to the Euler equations could dissipate energy, provided the velocity field was sufficiently irregular. We will refer to this phenomenon as \emph{inviscid anomalous dissipation}.  Onsager predicted $\frac{1}{3}$–H\"older continuity to be the critical regularity threshold for energy conservation, and that observation has guided much of the theoretical progress since. It is known that any Euler weak solution with H\"older exponent $>1/3$ (or in suitable Besov spaces) conserves energy (see \cite{Eyink1994, CET94,CCFS08}), whereas rougher solutions can in principle dissipate energy.

The first examples of such rough solutions were constructed by Scheffer (\cite{Scheffer93}) and Shnirelman (\cite{Shnirelman98}) in the 1990s. In \cite{CDL13}, De~Lellis and Sz\'ekelyhidi introduced a method of convex integration into fluid dynamics, which allowed to produce wild non-conservative solutions with better regularity. This program finally led to Isett’s construction of $C^{1/3-}_{x,t}$ solutions with inviscid anomalous dissipation (\cite{Ise18}). More recent constructions included dissipative solutions, intermittent examples in $H^{1/2-}$, and two-dimensional constructions \cite{MR4649134, BDLS16,NV23,brue2024flexibilitytwodimensionaleulerflows}. Despite these advances, solutions of the Euler equations obtained via convex integration are still not known to arise as vanishing-viscosity limits of physical Navier–Stokes solutions. 
In addition, solution to the Navier-Stokes equations obtained via convex integration do not produce the Kolmogorov dissipation range. This raises a natural question: Can anomalous dissipation in the Euler equations be realized through the vanishing-viscosity limit of Navier–Stokes solutions?

In this work, we address this question by studying the three-dimensional incompressible Navier–Stokes equations (NSE) on the periodic domain $\T^3$, and constructing vanishing-viscosity sequences that exhibit prescribed anomalous dissipation behavior. The NSE are given by the system
\begin{subequations}
\label{eq:NSE}
\begin{align}
\partial_t u^\nu 
+ u^\nu \cdot \nabla u^\nu 
+ \nabla p^\nu 
&= f^\nu + \nu \Delta u^\nu, 
\label{eq:NSE:1}
\\
\nabla \cdot u^\nu 
&=0
,
\label{eq:NSE:2}\\
u^\nu(0) &= u_{\textit{in}} \label{eq:NSE:3}
\end{align}
\end{subequations} \\
on the physical domain $\T^3 \times [0, T] $ and $T > 0$. Here $u^\nu(x,t)$ is the velocity field, $p^\nu$ the pressure, and $f^\nu(x,t)$ is an external force. The positive constant $\nu$ in equations \eqref{eq:NSE} stands for the viscosity coefficient. We consider the initial condition $u_{\textit{in}} \in L^2$ which is divergence free. For smooth solutions $u^\nu$ of the above equations on time interval $[0, T] $, the energy equality is satisfied
\begin{equation} \label{energy_eq_NSE}
\|u^{\nu}(t)\|_{L^2}^2 =  \|u_{\mathrm{in}}\|_{L^2}^2 - 2\nu\int_{0}^{t} \|\nabla u^{\nu}(\tau)\|_{L^2}^2 \, d\tau + 2\int_{0}^{t} (f^{\nu},u^{\nu}) \, d\tau,
\end{equation}\\
for each $0 < t \leq T$, where $(\cdot,\cdot)$ stands for the standard $L^2(\T^3)$ inner product. 

We are interested in the \emph{vanishing viscosity limit}: suppose we have a family of forces $\{f^\nu\}_{\nu > 0}$ which are smooth and satisfy
\[
f^\nu \to f \quad \text{in } L^1([0,T];L^2(\T^3)) \cap L^2([0,T];H^1(\T^3)),
\]
as $\nu \to 0$. We then call this $f$ the \emph{limiting force}. Suppose also that there is a family of solutions $u^\nu$ to \eqref{eq:NSE:1}--\eqref{eq:NSE:3} such that
\[
u^\nu \to u \quad \text{in } C_w([0,\infty);L^2),
\]
as $\nu \to 0$, and that $u$ is a weak solution of the Euler equations
\begin{subequations}
\label{eq:Euler}
\begin{align}
\partial_t u + u \cdot \nabla u + \nabla p &= f, 
\label{eq:Euler:1}\\
\nabla \cdot u &= 0, 
\label{eq:Euler:2}\\
u(0) &= u_{\textit{in}}. 
\label{eq:Euler:3}
\end{align}
\end{subequations}
We will call $u$ the vanishing viscosity limit of $u^\nu$, or the limiting solution.

\begin{Definition}\label{def:weak_sol_Eul}
    We say $u \in C_w([0,T];L^2)$ is a weak solution to the Euler equations $\eqref{eq:Euler}$ if for any smooth divergence free vector fields $\phi(x,t) \in C^\infty_c([0,T];\T^3)$ we have that 
    \begin{equation}\label{def:weak_sol_Eul_1}
    \int_0^T \int_{\T^3} u \cdot \p_t \phi + (u \otimes u ) \cdot \nabla \phi dxdt = \int_{\T^3} u(0,x) \phi(0,x) dx.
     \end{equation}
   Moreover, $u$ is weakly divergence free, i.e.
   \[
       \int_{\T^3} u \cdot \nabla \psi dx = 0
   \]
for any smooth function $\psi(x)$ on $\T^3$.
\end{Definition}
\begin{Remark}
    For NSE we have a similar definition, in fact, we say $u$ is a weak solution to the NSE $\eqref{eq:NSE}$ if $u$ is weakly divergence free and satisfies 
     \begin{equation*}
    \int_0^T \int_{\T^3} u \cdot \p_t \phi + (u \otimes u ) \cdot \nabla \phi + \nu u \cdot \Delta \phi dxdt = 0, \ \ \ \forall \phi \in C_c^\infty ([0,T],\T^3).
     \end{equation*}
\end{Remark}

Given a vanishing-viscosity sequence $u^\nu \to u$ as above, it is natural to compare the energy balances of $u^\nu$ and $u$. 
\begin{Definition}\label{Def_DA}
Let $t > 0$, we define the {\bf viscous dissipation limit} by
\begin{equation}\label{Dissipation_limit}
D(t) := \limsup_{\nu \to 0} \bigg ( 2\nu \int_0^t\twonorm{\nabla u^\nu(s)}^2 ds, 
\bigg )
\end{equation}\\
and the $\textbf{energy limit}$ by 
\begin{equation}\label{Energy_limit}
E(t) := \liminf_{\nu \to 0}{\twonorm{u^\nu(t)}^2}.
\end{equation}
The $\textbf{work}$ done by the limiting force $f$ is defined by
\begin{equation}\label{Work_by_force}
W(t) := \lim_{\nu \to 0} 2\int_0^t (f^{\nu},u^{\nu}) \, d\tau = 2\int_0^t (f,u)d\tau.
\end{equation}
\end{Definition}
The equality of the limit in $\rf{Work_by_force}$ follows from the convergence of the force $f^\nu \to f$ in $L^1([0,T];L^2)$ and the convergence of the solutions $u^\nu \to u$ in $C_w([0,T];L^2)$. Therefore if we take the limit of the energy equality \eqref{energy_eq_NSE} as $\nu \to 0$ and use the weak convergence of $u^\nu$ to $u$, then we obtain
\begin{equation} \label{eq:Intro_Enegry_Balance_1}
0 \leq \twonorm{u(t)}^2 \leq E(t)= \twonorm{u_\initial}^2 +  W(t) -  D(t) ,
\end{equation}
and hence 
\begin{equation} \label{eq:Intro_Energy_Balance_2}
0\leq D(t) \leq \twonorm{u_\initial}^2 - \twonorm{u(t)}^2  + W(t).
\end{equation}
In particular, $D(t)$ is always nonnegative, consistent with the physical fact that energy cannot be created in the limit of vanishing viscosity. The quantity $D(t)$ represents the total energy dissipated by viscosity (up to time $t$) in the vanishing-viscosity limit, while $E(t)$ and $W(t)$ represent the limits of the kinetic energy and work.  Any strict inequality in \eqref{eq:Intro_Energy_Balance_2} indicates a breakdown of energy conservation in the limit Euler dynamics. We can now formalize the notions of anomalous dissipation in this context (cf.\ \cite{cheskidov2023dissipation}):

\begin{Definition}\label{Def_DA_and_AD}
Let $D(t)$ and $W(t)$ given by $\rf{Dissipation_limit}$ and $\rf{Work_by_force}$ respectively, we say the weak solution $u(t)$ to \eqref{eq:Euler} exhibits $\textbf{inviscid anomalous dissipation}$ on $[0,t]$ if
\[
 \twonorm{u_\initial}^2 -\twonorm{u(t)}^2  +W(t) > 0.
\]
We say that a family $\{u^\nu\}_{\nu > 0}$ of solutions to \eqref{eq:NSE}  exhibit $\textbf{viscous dissipation anomaly}$ on $[0,t]$ if $D(t)>0$.
\end{Definition}

\noindent From $\eqref{eq:Intro_Energy_Balance_2}$ we see that viscous dissipation anomaly directly implies invicid anomalous dissipation. However, the converse is not true, see \cite{cheskidov2023dissipation} for the counterexample. This counterexample is constructed using the approach of Bru\`e and De Lellis in \cite{BDL23}. The idea of the scheme to construct a two-dimensional vector field $v^\nu$ that mixes a scalar $\theta^\nu$, which plays a role of the third component of the constructed solution $u^\nu$ to the \eqref{eq:NSE}, i.e. a $2\haf$ dimensional solution. In this paper, we will explore further on such construction by using an exponential time scale instead of a polynomial one in \cite{BDL23} and \cite{cheskidov2023dissipation}.

\begin{Remark} Note that not every inviscid anomalous dissipative solutions is a vanishing viscosity limit. Indeed, one can construct inviscid anomalous dissipative solutions via convex integrations, see for example \cite{Ise18,BDLS16,NV23}. Moreover, a highly intermittent 2D construction in \cite{brue2024flexibilitytwodimensionaleulerflows}, is known not to be a vanishing viscosity limit, which follows from \cite{CLL16}.
The question of whether a vanishing viscosity limit solution $u$ of the equation $\rf{eq:Euler}$ can also be obtained by convex integration is remained opened.
\end{Remark}

\noindent
{\bf Previous construction of viscous dissipation anomaly.}
Consider the passive scalar advection--diffusion equation on $\T^3$,
\[
\partial_t \theta^{\kappa} + v\cdot\nabla \theta^{\kappa} = \kappa \Delta \theta^{\kappa},
\]
where $v$ is divergence-free and (in the anomalous regime) is typically only sub-Lipschitz in space.
We say that $\{\theta^{\kappa}\}_{\kappa>0}$ exhibits \emph{anomalous diffusion} (or a \emph{viscous dissipation anomaly} of scalar variance) on $[0,1]$ if 
\[
\limsup_{\kappa\to0}\;\kappa\int_0^1\|\nabla\theta^{\kappa}(t)\|_{L^2(\T^3)}^2\,dt >0.
\]
If $v$ is uniformly Lipschitz, then the left-hand side vanishes. Hence anomalous diffusion requires sufficiently rough drifts.
This effect is a central prediction of the Obukhov--Corrsin theory and is modeled by the stochastic Kraichnan framework \cite{Obukhov1949,Corrsin1951,Kraichnan1968}.
Rigorous deterministic examples include the drift smooth except at one time constructed by Drivas--Elgindi--Iyer--Jeong \cite{DrivasElgindiIyerJeong2022} (see also the further refinement of the result \cite{CCS23, EL24}) and a recent example of Armstrong--Vicol yielding order-one dissipation for periodic H\"older drifts $v\in C^0_t C^\alpha_x$ with any $\alpha<1/3$ \cite{AV25}.
Also, Burczak, Sz\'ekelyhidi, and Wu \cite{BurczakSzekelyhidiWu2024Euler} show that anomalous dissipation can occur with $v$ itself a weak Euler solution of regularity $C^{1/3-}$; see also the 2025 overview \cite{ArmstrongVicol2025Overview}.

In \cite{BDL23}, Bru\`e and De Lellis have shown that viscous dissipation anomaly can be realized for the forced $3$D Navier--Stokes equations within the $(2+\frac12)$-dimensional ansatz $u^\nu=(v^\nu,\theta^\nu)$ and $f^\nu=(g^\nu,0)$, so that $v^\nu$ solves a forced $2$D Navier--Stokes system while $\theta^\nu$ satisfies a passive advection--diffusion equation driven by $v^\nu$. 
The forcing $g^\nu$ is tuned so that $v^\nu$ remains uniformly regular (in particular bounded in $C^\alpha$ for every $\alpha<1$) while still generating a quasi-self-similar mixing dynamics for $\theta^\nu$ that transfers scalar variance to smaller and smaller scales before diffusion becomes effective. 
This mixing mechanism is implemented by iterating smooth self-similar mixers (in the spirit of Alberti--Crippa--Mazzucato~\cite{ACM16}) on a sequence of shrinking time intervals accumulating at a finite time. 
As a consequence, the scalar dissipation remains order one, i.e.\ $\liminf_{\nu\to0}\nu\int_0^1\|\nabla\theta^\nu(t)\|_{L^2}^2\,dt>0$, and since $\theta^\nu$ is a component of $u^\nu$ this yields a genuine kinetic-energy dissipation anomaly for the $3$D Navier--Stokes solutions in the vanishing-viscosity limit.

The dissipation–anomaly mechanism introduced by Bru\`e and De Lellis \cite{BDL23} has been refined in recent works.  Bru\`e, Colombo, Crippa, De Lellis and Sorella combined the $2\frac12$–dimensional Navier–Stokes architecture of \cite{BDL23} with the quantitative passive–scalar mixing scheme of Colombo, Crippa and Sorella \cite{CCS23} to construct vanishing–viscosity solutions of the forced three–dimensional Navier–Stokes equations which exhibit anomalous dissipation while remaining uniformly bounded in $L^3_t C^{1/3-\varepsilon}_x$ for any $\varepsilon>0$ \cite{BCCD24}.  Building on the same framework, Johansson and Sorella answered further questions of \cite{BDL23} by producing four–dimensional forced Navier–Stokes flows whose dissipation measures have a nontrivial absolutely continuous component in time, via a $3+\tfrac12$–dimensional reduction and new anomalous–diffusion estimates \cite{JohanssonSorella26}.  In a different direction, in \cite{cheskidov2023dissipation}, a quasi–self–similar $2\frac12$–dimensional construction was used to generate families of Navier–Stokes solutions displaying a wide range of vanishing–viscosity scenarios—total and partial viscous dissipation anomaly, inviscid anomalous dissipation without viscous dissipation anomaly, and examples with absolutely continuous dissipation measures—thus showing the flexibility of the Bru\`e–De Lellis paradigm.

\subsection*{Main contributions}
We prove three results:
\begin{itemize}
\item[(i)] An adapted flux criterion for $2\frac{1}{2}$-dimensional Euler solutions (Theorem~\ref{main_Theorem_1}), guaranteeing energy conservation under sharp Besov control on $(v,\rho)$.
\item[(ii)] A vanishing-viscosity sequence with \emph{total} viscous anomalous dissipation, whose limit remains borderline in Onsager scale (Theorem~\ref{main_Theorem_2}).
\item[(iii)] A fully three-dimensional construction with rough forcing, zero anomalous work and total viscous dissipation, whose limiting solution lies in the Onsager borderline space (Theorem~\ref{main_Theorem_3}).
\end{itemize}

\subsection*{Methods and novelty}
We refine the perfect-mixing paradigm of \cite{BDL23,cheskidov2023dissipation} by introducing an \emph{exponential} time scale, enabling precise dissipation bookkeeping while preserving Onsager-critical spectral bounds. We also exploit a weighted Besov characterization near the critical scale.

\subsection*{Organization}
Section~\ref{Sec_Prelim} collects notation and analytical tools. Section~\ref{Sec_Main_Thm} states our main results. Sections~\ref{sec:proof_main_Theorem_1}--\ref{sec:proof_main_Theorem_3} provide the corresponding proofs.


\section{Preliminaries}\label{Sec_Prelim}
\subsection{Notations}\label{Notation} We use the notation $A \lesssim B$ if there exists an absolute constant $C > 0$ such that $A \leq CB$. If there exists $C_1$ and $C_2$ such that $C_1B \leq A \leq C_2 B$, we write $A \sim B$.

\subsection{Cartesian $2\frac{1}{2}$-dimensional solutions}

The so-called $2\frac{1}{2}$-dimensional (or ``two-and-a-half dimensional'') flows will be used in the proofs of Theorems~\ref{main_Theorem_1} and~\ref{main_Theorem_2}. Such flows are defined as three-dimensional velocity fields that depend only on two spatial variables. Specifically, consider
\begin{align}\label{def_2.5D_flow}
    u(x_1,x_2,x_3) = 
    \begin{pmatrix}
        u_1(x_1,x_2) \\
        u_2(x_1,x_2) \\
        u_3(x_1,x_2)
    \end{pmatrix}
    :=
    \begin{pmatrix}
        v(x_1,x_2) \\
        \rho(x_1,x_2)
    \end{pmatrix},
\end{align}
where $v(x_1,x_2) = (u_1,u_2)$ is a two-dimensional vector field and $\rho(x_1,x_2)$ is a scalar function. Under this ansatz, the three-dimensional incompressible Euler equations reduce to
\begin{align}
    \partial_t v + v \cdot \nabla v + \nabla p &= 0, \label{eq:2.5D_flow_1}\\
    \partial_t \rho + v \cdot \nabla \rho &= 0. \label{eq:2.5D_flow_2}
\end{align}
An important observation is that the third component $\rho$ satisfies a purely advective (transport) equation. Similarly, for the three-dimensional incompressible Navier–Stokes equations, the same structure persists:
\begin{align}
    \partial_t v + v \cdot \nabla v + \nabla p &= \nu \Delta v, \label{eq:2.5D_flow_3}\\
    \partial_t \rho + v \cdot \nabla \rho &= \nu \Delta \rho. \label{eq:2.5D_flow_4}
\end{align}
\subsection{Littlewood--Paley decomposition}\label{LP_Intro}
We briefly summarize the Littlewood--Paley theory; for more comprehensive background, we refer to the books by Bahouri, Chemin, and Danchin~\cite{bahouri2011fourier} and by Grafakos~\cite{grafakos2008classical}. To begin, let $\chi \in C_0^\infty(\R^3)$ be a nonnegative, radial function such that 
\begin{equation} \label{eq:xi}
\chi(\xi) :=
\begin{cases}
1, & \text{for } |\xi| \le \tfrac{3}{4},\\
0, & \text{for } |\xi| \ge 1.
\end{cases}
\end{equation}
We then define
\[
\varphi(\xi) := \chi(\xi/2) - \chi(\xi),
\]
and set
\begin{equation*}
\varphi_q(\xi) :=
\begin{cases}
\varphi(\eta_q^{-1}\xi), & q \ge 0,\\
\chi(\xi), & q = -1,
\end{cases}
\end{equation*}
where $\eta_q = 2^q$. The family $\{\varphi_q\}_{q \ge -1}$ thus forms a dyadic partition of unity in the frequency space. Given a tempered distribution $u$ on the three-dimensional torus $\T^3 = [0,L]^3$, we define its $q$th Littlewood--Paley projection by
\[
  \Delta_q u(x) := \sum_{k\in\Z^3} \hat{u}(k)\,\varphi_q(k)\,e^{i\frac{2\pi}{L} k \cdot x},
\]
where $\hat{u}(k)$ denotes the $k$th Fourier coefficient of $u$.  
Note that $\Delta_{-1} u = \hat{u}(0)$ corresponds to the mean (low-frequency) part of $u$.  
With this notation, we have the decomposition
\[
u = \sum_{q=-1}^\infty \Delta_q u,
\]
in the sense of distributions. We also define the low-frequency cut-off operator
\[
\Ss_q u := \sum_{p=-1}^q \Delta_p u.
\]
In this paper (except for section \ref{Sec_Second_construction}), however, we do not adopt the standard dyadic partition of unity. Instead, we use a frequency sequence defined by
\begin{equation}\label{def:lambdaq}
\lambda_q := 5^q,
\end{equation}
that is, we replace $\eta_q$ above with $\lambda_q$. The corresponding Littlewood--Paley decomposition remains valid after rescaling. \\\\
We recall here the useful Bernstein's inequality for each block of the Littlewood-Paley decomposition.
\begin{Lemma}\label{lemma:Bernstein}(Bernstein's inequality) 
Let $d$ be the spacial dimension, $r\geq s\geq 1$ and $k\geq 0$. Then for all tempered distributions $u$, 
\bg\label{Bern1}
\lambda_q^{k}\|\Dq u\|_{r} \lesssim \|\nabla^k \Dq u\|_{r}\lesssim \lambda_q^{k+d(\frac{1}{s}-\frac{1}{r})}\|\Dq u\|_{s}.
\ed
\end{Lemma}
\vspace{1em}
\subsection{Besov spaces and Onsager's critical energy classes}

For $s \in \R$ and $1 \le p, r \le \infty$, the Besov space $\Bb^s_{p,r}$ consists of all tempered distributions $u$ such that
\begin{equation}\label{def_Besov1}
\|u \|_{\Bb^s_{p,r}} := 
\left( \sum_{q=-1}^\infty \lambda_q^{rs} \| \Delta_q u \|_{L^p(\T^3)}^r \right)^{1/r}
< \infty,
\end{equation}
for $r < \infty$. In the case $r = \infty$, we set
\begin{equation}\label{def_Besov2}
\|u \|_{\Bb^s_{p,\infty}} := 
\sup_{q \ge -1} \left( \lambda_q^{s} \| \Delta_q u \|_{L^p(\T^3)} \right) < \infty.
\end{equation}

\noindent We further define the following subspace of $\Bb^{1/3}_{3,\infty}$:
\begin{equation}\label{def_Besov3}
\Bb^{1/3}_{3,c_0} := \left \{ u \in \Bb^{1/3}_{3,\infty} : \limsup_{q \to \infty}  \lambda_q^{1/3} \| \Delta_q u \|_{L^3} = 0 \right \}
\end{equation}

\vspace{0.5em}
\noindent A celebrated result by Constantin, E, and Titi~\cite{CET94} states that if a weak solution $u$ of the incompressible Euler equations satisfies
\[
u \in L^3\big([0,T]; \Bb^{\frac{1}{3}+}_{3,\infty}\big),
\]
then $u$ conserves energy, i.e., it does not exhibit inviscid anomalous dissipation. Conversely, via convex integration techniques, it was shown by Isett~\cite{Ise18} that there exists weak solutions
\[
u \in C^0\big([0,T]; C^{\frac{1}{3}-}\big)
\]
that do exhibit inviscid anomalous dissipation. Here we use the standard shorthand notation:
\begin{equation}\label{def_Besov_pm}
\Bb^{\frac{1}{3}+}_{3,\infty} := \bigcup_{\varepsilon > 0} \Bb^{\frac{1}{3}+\varepsilon}_{3,\infty},
\qquad
\Bb^{\frac{1}{3}-}_{3,\infty} := \bigcap_{\varepsilon > 0} \Bb^{\frac{1}{3}-\varepsilon}_{3,\infty}.
\end{equation}
In~\cite{CCFS08}, Cheskidov, Constantin, Friedlander, and Shvydkoy introduced a sharper sufficient condition ensuring energy conservation for weak Euler solutions.
  
\begin{Theorem}[\cite{CCFS08}]\label{Thm:energy_conservation_CCFS08}
Suppose that $u$ is a weak solution to the Euler equations~\eqref{eq:Euler} on $[0,T]$ satisfying
\begin{equation}\label{Thm:energy_conservation_CCFS08_1}
\lim_{q \to \infty} \int_0^T \lambda_q \| \Delta_q u(t) \|_{L^3}^3 \, dt = 0.
\end{equation}
Then the kinetic energy of $u$ is conserved on $[0,T]$.  
In particular, if $u$ is a weak solution such that
\[
u \in L^3\big([0,T]; \Bb^{1/3}_{3,c_0}(\T^3)\big),
\]
then $u$ conserves energy.
\end{Theorem}

\noindent To understand the meaning of condition~\eqref{Thm:energy_conservation_CCFS08_1}, we recall the truncated energy equality
\begin{equation}\label{eq:truncated_energy_equality}
\frac{1}{2}\frac{d}{dt} \| \Ss_q u(t) \|_{L^2}^2 
=  \int_{\T^3} \Ss_q (u \otimes u) \cdot \nabla \Ss_q u \, dx 
=: \Pi_q,
\end{equation}
where $\Pi_q$ is referred to as the 'energy flux' through the frequency shell $q$.  
Integrating over time from $0$ to $T$ gives
\[
\| \Ss_q u(T) \|_{L^2}^2 - \| \Ss_q u(0) \|_{L^2}^2
= 2 \int_0^T \Pi_q(t) \, dt.
\]
The proof of Theorem~\ref{Thm:energy_conservation_CCFS08} relies on the estimate
\[
\int_0^T |\Pi_q(t)| \, dt 
\lesssim \int_0^T \lambda_q \| \Delta_q u(t) \|_{L^3}^3 \, dt,
\]
from which condition~\eqref{Thm:energy_conservation_CCFS08_1} follows.

\vspace{1em}
\noindent In this paper, we will construct dissipative weak solutions of the Euler equations that are sharp in the sense of Theorem~\ref{Thm:energy_conservation_CCFS08}.  
More precisely, our weak solution $u$ satisfies
\begin{equation}\label{eq:sharp_Onsagers_condition}
\limsup_{q \to \infty} 
\int_0^T \lambda_q \| \Delta_q u(t) \|_{L^3}^3 \, dt 
\lesssim 1.
\end{equation}
We also introduce here the following weighted Besov space:
\begin{Definition}\label{def_weighted_Besov}
Let $\{a_n\}_{n \in \N} \in \ell^1_+$, where $\ell^1_+$ denotes the collection of non-negative sequences in $\ell^1$, we define the following weighted Besov space $\Bb^{s,a_n}_{p,\infty}$ by
\begin{equation}\label{def_weighted_Besov_1}
\|u \|_{\Bb^{s,a_n}_{p,\infty}} :=  \sum_{q = -1}^\infty \Big ( a_{q+2}\l^s_q \| \Delta_q u \|_{L^p(\T^3)} \Big) < \infty.
\end{equation}
In addition, denote $\Tt$ by the subfamily of $\ell^1_+$ such that 
\begin{equation}\label{def_weighted_Besov_2}
\Tt := \Big \{ \{ a_n \}_{n\in \N} \in \ell^1: 0< a_n < 1, a_{n+1} < a_n \text{ and }\lim_{n \to \infty}{\frac{a_{n+1}}{a_n}} = 1, \ \  n \in \N  \Big \}. 
\end{equation}
\end{Definition}
\begin{Proposition}\label{prop_weighted_Onsager_class}
For any sequence $\{a_n \}_{n\in \N} \in \Tt$, we have the following inclusions:
\begin{equation}\label{prop_weighted_Onsager_class_1}
\Bb^\frac{1}{3}_{3,\infty} \subset \Bb^{\frac{1}{3}, a_n}_{3,\infty} \subset \Bb^{\frac{1}{3}-}_{3,\infty}.
\end{equation}
Furthermore, we have 
\begin{equation}\label{prop_weighted_Onsager_class_2}
\Bb^\frac{1}{3}_{3,\infty} = \bigcap_{b_n \in l_{+}^1} \Bb^{\frac{1}{3},b_n}_{3,\infty}.
\end{equation}
\end{Proposition}
\begin{proof}
Suppose that $u \in \Bb^\frac{1}{3}_{3,\infty}$, then 
\[
\sum_{q \geq -1} \Big ( a_{q+2}\l^{\frac{1}{3}}_q \threenorm{\Dq u} \Big) \leq \Big(\sum_{k=1}^\infty a_k \Big) \Big (\sup_{q \geq 1} \l^{\frac{1}{3}}_q \threenorm{\Dq u} \Big) < \infty, 
\]
since $\{a_n\}_{n\in \N} \in \Tt$. This implies the first inclusion in \eqref{prop_weighted_Onsager_class_1}. Fix $\e > 0$, the second inclusion is obtained by 
\begin{align*}
    \sup_{q \geq -1} \Big(  \l_q^{\frac{1}{3}-\e} \threenorm{\Dq u} \Big) &< \sup_{q \geq -1} \bigg( \Big( \frac{a_{q+2}}{a_{q+3}} \Big) \l^{\frac{1}{3}-\e}_q \threenorm{\Dq u} \bigg) \\ 
    & \leq \Big(\sum_{q\geq -1} a_{q+2} \l_q^{\frac{1}{3}}  \threenorm{\Dq u}\Big) \Big( \sup_{q \geq -1} (a_{q+3} \l^\e_q)^{-1}\Big) < \infty.
\end{align*}
The boundedness of the sequence $b_q := (a_{q+3}\l^\e_q)^{-1}$ follows directly from 
\[
\lim_{q \to \infty} \frac{b_{q+1}}{b_q} = \l^{-\e} < 1.
\]
To establish $\eqref{prop_weighted_Onsager_class_2}$, first notice that from $\eqref{prop_weighted_Onsager_class_1}$ we immediately have the forward inclusion. Now suppose if $u \notin B^{\frac{1}{3}}_{3,
\infty}$, then 
\[
\sup_{q\geq -1} \Big ( \l^\frac{1}{3}_q \threenorm{\Dq u} \Big) = \infty,
\]
hence for $k \in \N$ we can find a subsequence $\{ q_k \}_{k \in \N}$ such that 
\[
\l^\frac{1}{3}_{q_k} \threenorm{\Delta_{q_k} u } > k^3.
\]
Now take $b_n \in \ell^1_+$ such that
\begin{align*}
b_n = \begin{cases}
\frac{1}{k^2}, \qquad n = q_k, \\
0, \qquad otherwise.
\end{cases}
\end{align*}
Therefore we have that 
\[
\|u \|_{B^{\frac{1}{3},b_n}_{3,\infty}} = \sum_{k =1}^\infty \Big( b_{q_k} \l^\frac{1}{3}_{q_k} \threenorm{\Delta_{q_k} u } \Big) = \sum_{n =1 }^\infty \frac{1}{n} = \infty
\]
and therefore 
\[
u \notin \bigcap_{b_n \in \ell^1_+} \Bb^{\frac{1}{3},b_n}_{3,\infty}.
\]
\end{proof}

\section{main Theorem}\label{Sec_Main_Thm}
\noindent We present the following main results:
\begin{Theorem}\label{main_Theorem_1}
Let $u$ be a $2\haf$-D weak solutions of the form \eqref{def_2.5D_flow} to the Euler equations \eqref{eq:Euler} on $[0,T]\times \R^2$. If u satisfies the following condition:
\begin{equation}\label{main_Theorem_1_1}
\lim_{q \to \infty} \int_0^T \l_q \|\Dq v(t) \|^3_{L^3} dt = 0
\end{equation}
and
\begin{equation}\label{main_Theorem_1_2}
\lim_{q \to \infty} \int_0^T \l_q \|\Dq v(t) \|_{L^3} \|\Dq \rho (t) \|^2_{L^3} dt = 0
\end{equation}
with the external force $f \in L^1([0,T];L^2(\R^2))$, then the energy of $u$ is conserved on $[0,T]$. In particular, if $v \in L^3([0,T];B^{\onethird}_{3,c_0}(\R^2))$ and $\rho \in L^3([0,T];B^{\onethird}_{3,\infty}(\R^2))$, then $u$ conserves energy on $[0,T]$.
\end{Theorem}
\begin{remark}
The condition \eqref{main_Theorem_1_2} is an adaptation to the condition $\eqref{Thm:energy_conservation_CCFS08_1}$ in the case of $2\haf$-dimensional solutions. The idea follows from the observation of the nonlinear flux term of the full Euler equations
\[
\int_{\R^3} (u \cdot \nabla) u \cdot u dx 
\]
is reduced to 
\[
\int_{\R^2} (v\cdot \nabla)v \cdot v dx + \int_{\R^2} (v\cdot \nabla\rho) \rho dx.
\]
\end{remark}
\vspace{1em}
\begin{Theorem}\label{main_Theorem_2}
Given any $\{a_n\}_{n \in \N} \in \Tt$ defined in \eqref{def_weighted_Besov_2}, there is a countable family of smooth solutions to $(\ref{eq:NSE})$, denoted by $\{u^{\nu_n} (t)\big \}_{\nu_n > 0}$, on time interval $[0,1]$, with a positive sequence of viscosity $\nu_n \to 0$ as $n \to \infty$ under the presence of the external forces $f^{\nu_n} \to f $ in $L^1([0,1]; L^2(\T^3))$. Furthermore, the family of solutions satisfy some initial condition $u^{\nu_n}(0)= u_\initial \in L^2(\T^3)$, $u^{\nu_n} \to u$ strongly in $L^2$ at time $t \in \left[0,1 \right )$ and weakly at $t = 1$ for some $2\haf$-dimensional weak solution $u\in L^3([0,1];\Bb^{\frac{1}{3},a_n}_{3,\infty}(\T^3))$ to \eqref{eq:Euler}. The weak solution $u$ of the form $\eqref{def_2.5D_flow}$ also satisfies that
\begin{equation}\label{main_Theorem_2_1}
\limsup_{q \to \infty}\int_0^1 \l_q \threenorm{\Dq v(t)} \threenorm{\Dq \rho(t)}^2 dt \les 1
\end{equation}
and 
\begin{equation}\label{main_Theorem_2_2}
\lim_{q \to \infty}\int_0^1 \l_q \threenorm{\Dq v(t)}^3 dt =0.
\end{equation}
In addition, the family $\{u^{\nu_n}\}_{\nu_n > 0}$ exhibits total viscous anomalous dissipation, i.e. 
\begin{equation}\label{main_Theorem_2_3}
\lim_{n \to \infty} 2\nu_n \int_0^1 \twonorm{\nabla u^{\nu_n}(t)}^2 dt = \twonorm{u_\initial}^2.
\end{equation}
\end{Theorem} 
\vspace{1em}
\begin{Theorem}\label{main_Theorem_3}
There is a countable family of smooth solutions to $(\ref{eq:NSE})$ under the presence of the external forces, denoted by $\{u^{n} (t),f^{\nu_n}(t) \big \}_{\nu_n > 0}$, on time interval $[0,T]$ for some $T > 0$, with a positive sequence of viscosity $\nu_n \to 0$ as $n \to \infty$, such that \\
\begin{enumerate}
\item $u^n \to u$ in $L^\infty([0,T];L^2(\T^3))$ and $f^{\nu_n} \to f$ in $L^1([0,T];L^r(\T^3))$ with any $r < 2$, for some $(u,f)$. \label{main_Theorem_3_1}\\
\item The pair $(u,f)$ is a weak solution to \eqref{eq:Euler} satisfying that $u \in L^3([0,T];B^{\frac{1}{3},c_n}_{3,\infty})$ for any sequence $\{c_n\}_{n \in \N} \in \ell^1_+$. In addition, $u^n$ and $u$ shares the same given initial condition, $u^n(0) = u(0) = u_\initial$. \label{main_Theorem_3_2} \\
\item \label{main_Theorem_3_3} It holds that 
\begin{equation}\label{main_Theorem_3_3_1}
\limsup_{q \to \infty}\int_0^T \tilde{\l}_q \threenorm{\Dq u(t)}^3 dt \les 1,
\end{equation}
where $\tilde{\l}_q := 2^q$ and the amount of anomalous work of $f$ is zero. \\
\item \label{main_Theorem_3_4} The family of NSE solutions $\{ u^n, f^{\nu_n} \}_{n \in \N,\nu_n > 0}$ exhibits total viscous anomalous dissipation, i.e.
\begin{equation}\label{main_Theorem_3_4_1}
\lim_{n \to \infty} 2 \nu_n \int_0^T \twonorm{\nabla u^n(t)}^2 dt = \twonorm{u_\initial}^2
\end{equation}
and performs zero work:
\begin{equation}\label{main_Theorem_3_4_2}
\int_0^T \<f^{\nu_n}, u^n \> dt = 0, \ \ \forall n \in \N.
\end{equation}
\end{enumerate}

\end{Theorem}
\vspace{1em}

\noindent The rest of the paper is dedicated to the proof of the above Theorems. The proof of Theorem $\ref{main_Theorem_1}$ follows very closely to the argument in \cite{CCFS08}. Theorem $\ref{main_Theorem_2}$ relies on the construction of $2\haf$-dimensional solutions and perfect mixing from \cite{BDL23} and \cite{cheskidov2023dissipation}.  Theorem \ref{main_Theorem_3} is based on the result of \cite{CL21}, which gives a full three dimensional solution under a rougher force.

\section{proof of Theorem \ref{main_Theorem_1}}\label{sec:proof_main_Theorem_1}

\noindent Firstly, without loss of generality we take the external force to be zero. From the result of \cite{CCFS08} and the assumption \eqref{main_Theorem_1_1}, it follows that the weak solutions $v$ to the $2D$ Euler equations conserves energy, i.e. 
\[
\twonormRtwo{v(T)}^2 = \twonormRtwo{v(0)}^2. 
\]
Hence it suffices to show that 
\[
\twonormRtwo{\rho(T)}^2 = \twonormRtwo{\rho(0)}^2 .
\]
Following the argument of \cite{CCFS08}, we denote 
\begin{align}
J &= \Ff^{-1} \varphi, \ \ \ \widetilde{J} = \Ff^{-1} \chi, \label{proof:energy_conservation_twohalf_1} \\
\Delta_q u &= \powermone{\Ff} (\varphi_q \Ff u) = \l^2_q \int_{\R^2} J(\l_q y) u(x-y) dy, \ \ q \geq 0, \label{proof:energy_conservation_twohalf_2} \\ 
\Delta_{-1} u &= \Ff^{-1} (\chi \Ff u) = \int_{\R^2} \wtj(y) u(x-y) dy, \label{proof:energy_conservation_twohalf_3} \\
\Pi_q &= \int_{\R^2} \Ss_q (v\rho) \cdot \nabla \Ss_q \rho(x,t) dx,  \label{proof:energy_conservation_twohalf_4} \\
\Ss_q u &= \sum_{p = -1}^q \Delta_p u,\label{proof:energy_conservation_twohalf_5}
\end{align}
where $\varphi, \varphi_q$ and $\chi$ are defined in section \ref{LP_Intro}. Recall from \eqref{eq:2.5D_flow_2} that since $\rho(t)$ is a weak solutions, we can test the transport equation with $\Ss_q \Ss_q \rho$ and obtain
\begin{equation}\label{proof:energy_conservation_twohalf_6}
\haf \frac{d}{dt} \twonormRtwo{\Ss_q \rho(t)}^2 =  \int_{\R^2} \Ss_q (v\rho) \cdot \nabla \Ss_q \rho dx = \Pi_q.
\end{equation}
 We observe that 
\begin{equation}\label{proof:energy_conservation_twohalf_7}
\Ss_q (v \rho) = r_q (v,\rho) - (v-\Ss_q v)(\rho - \Ss_q \rho) + \Ss_q v \Ss_q \rho 
\end{equation}
with
\[
r_q (v,\rho) = \l^2_q \int_{\R^2} \wtj(\l_q y) \big(v(x-y) -v(x) \big) \big(\rho(x-y) - \rho(x) \big) dy.
\]
Substituting \eqref{proof:energy_conservation_twohalf_7} into the right hand side of \eqref{proof:energy_conservation_twohalf_6} yields 
\begin{align*}
\Pi_q &= \int_{\R^2} r_q (v,\rho) \cdot \nabla \Ss_q \rho dx + \int_{\R^2} \Ss_q v \Ss_q \rho \cdot \nabla \Ss_q \rho dx - \int_{\R^2} (v-\Ss_q v)(\rho - \Ss_q \rho) \cdot \nabla \Ss_q \rho dx \\
&= \int_{\R^2} r_q (v,\rho) \cdot \nabla \Ss_q \rho dx - \int_{\R^2} (v-\Ss_q v)(\rho - \Ss_q \rho) \cdot \nabla \Ss_q \rho dx := A+B.
\end{align*}
The middle term of the above vanishes due to the incompressibility. Now we estimate the term $A$ by:
\[
|A| = \bigg| \int_{\R^2} r_q (v,\rho) \cdot \nabla \Ss_q \rho dx \bigg| \leq \|r_q(v,\rho) \|_{L^{\frac{3}{2}}} \|\nabla \Ss_q \rho \|_{L^3},
\]
since 
\begin{align*}
\|r_q(v,\rho) \|_{L^{\frac{3}{2}}} &\leq \int \l^2_q|\wtj(\l_q y)| \bigg( \int |v(x-y) - v(x)|^{\frac{3}{2}} |v(x-y)-v(x)|^{\frac{3}{2}} dx \bigg)^{\frac{2}{3}} dy \\\\
&\leq \int \l^2_q |\wtj(\l_q y)| \|v(\cdot - y) - v(\cdot) \|_{L^3} \|\rho(\cdot - y) - \rho(\cdot) \|_{L^3} dy \\\\
&\leq \bigg( \int \l^2_q |\wtj(\l_q y)| \|v(\cdot - y) - v(\cdot) \|^2_{L^3}  dy \bigg)^\haf \bigg( \int \l^2_q |\wtj(\l_q y)| \|\rho(\cdot - y) - \rho(\cdot) \|^2_{L^3}  dy \bigg)^\haf 
\end{align*}
and 
\begin{align*}
\|v(\cdot -y) - v(\cdot) \|^2_{L^3} &\les \sum_{p \leq q} |y|^2 \l_q^2 \|\Delta_p v\|^2_{L^3} + \sum_{p > q} \|\Delta_p v \|^2_{L^3}\\ 
&\les \l^{\frac{4}{3}}_q |y|^2 \sum_{p \leq q} \l^{-\frac{4}{3}}_{q-p} d^2_q + \l^{-\frac{2}{3}}_q \sum_{p > q} \l^{\frac{2}{3}}_{q-p} d^2_q \\
&\les \Big( \l^{\frac{4}{3}}_q |y|^2 + \l^{-\frac{2}{3}}_q \Big) \big( K * d^2 \big) (q),
\end{align*}
where 
\begin{align}
d_q &:= \l^{\frac{1}{3}}_q \|\Delta_q v \|_{L^3}, \ \ \ d := \{d_q\}_{q \geq 1}, \label{proof:energy_conservation_twohalf_8} \\
K &:= \begin{cases}
\l^{\frac{2}{3}}_q, \ \ q \leq 0, \\
\l_q^{-\frac{4}{3}}, \ \ q > 0.
\end{cases} \label{proof:energy_conservation_twohalf_9}
\end{align}
Similarly for 
\[
\|\rho(\cdot - y) - \rho(\cdot) \|^2_{L^3} \les \Big( \l^{\frac{4}{3}}_q |y|^2 + \l^{-\frac{2}{3}}_q \Big) \big( K * \tilde{d}^2 \big) (q),
\] 
where 
\begin{equation}\label{proof:energy_conservation_twohalf_10}
\td_q := \l^{\frac{1}{3}}_q \|\Dq \rho \|_{L^3}, \ \ \ \td:= \{ \td_q \}_{q \geq 1}.
\end{equation}
Therefore
\begin{align*}
\bigg( \int \l^2_q |\wtj(\l_q y)| \|v(\cdot - y) - v(\cdot) \|^2_{L^3}  dy \bigg)^\haf &\les \big( K * d^2 \big)^\haf (q) \Bigg( \bigg( \int \l_q^2 |\wtj(\l_q y)| \l^{\frac{4}{3}}_q |y|^2 dy \bigg)^\haf + \l_q^{-\frac{1}{3}}\Bigg) \\
&\les \l^{-\onethird}_q \big( K * d^2 \big)^\haf (q) \Bigg(1+ \bigg( \int (\l_q |y|)^2 \l^2_q |\wtj(\l_q y)| dy \bigg)^\haf \Bigg) \\
&\les \l^{-\onethird}_q (K * d^2)^\haf (q)
\end{align*}
and likewise,
\[
\bigg( \int \l^2_q |\wtj(\l_q y)| \|\rho(\cdot - y) - \rho(\cdot) \|^2_{L^3}  dy \bigg)^\haf \les \l^{-\onethird}_q (K * \td^2)^\haf (q).
\]
Then 
\begin{align*}
|A| &\les \|r_q(v,\rho) \|_{L^{\frac{3}{2}}} \|\nabla \Ss_q \rho \|_{L^3}\\
&\les \l^{-\frac{2}{3}}_q (K * d^2)^\haf (q) (K * \td^2)^\haf (q) \Big( \sum_{p \leq q} \l^2_p \|\Delta_p \rho \|^2_{L^3} \Big)^\haf \\
&\les \l_q^{-\frac{2}{3}} (K * d^2)^\haf (q) (K * \td^2)^\haf (q) \Big( \sum_{p \leq q} \l_p^{\frac{4}{3}} \td^2_p \Big)^\haf \\
&\les (K * d^2)^\haf (q) (K * \td^2)^\haf (q) \Big( \sum_{p \leq q} \l_{q-p}^{-\frac{4}{3}} \td^2_p \Big)^\haf \les (K * d^2)^\haf (q) (K * \td^2). 
\end{align*}
The term $B$ is bounded similarly by 
\begin{align*}
|B| &\leq \int \Big| (v-\Ss_q v) (\rho - \Ss_q \rho) \cdot \Ss_q \rho \Big| dx \leq \|v - \Ss_q v \|_{L^3} \| \rho - \Ss_q \rho \|_{L^3} \|\nabla \Ss_q \rho \|_{L^3} \\
&= \Big(\sum_{p >q} \|\Delta_p v \|_{L^3} \Big) \Big(\sum_{p >q} \|\Delta_p \rho \|_{L^3} \Big) \Big(\sum_{p \leq q} \l^2_p \|\Delta_p \rho \|^2_{L^3} \Big)^\haf \\
&\les \l^{-\onethird}_q  \Big(\sum_{p >q} \l^\onethird_{q-p} d_q \Big) \l^{-\onethird}_q \Big(\sum_{p >q} \l^\onethird_{q-p} \td_p  \Big) \Big(\sum_{p \leq q} \l^{\frac{4}{3}}_p \td^2_p \Big)^\haf \\
&\les  (K^\haf * d) (q) (K^\haf * \td) (q) \Big(\sum_{p \leq q} \l^{-\frac{4}{3}}_{q-p} \td^2_p \Big)^\haf \les  (K * d^2)^\haf (q) (K * \td^2) (q). 
\end{align*}
Finally by \eqref{proof:energy_conservation_twohalf_6}, we have that for fixed integer $q \geq -1$,
\begin{align*}
\haf \twonormRtwo{\Ss_q \rho(T)}^2 - \haf \twonormRtwo{\Ss_q \rho(0)}^2 &= \int_0^T \Pi_q dt \les \int_0^T (|A| + |B|) dt \\ 
&\les \int_0^T (K * d^2)^\haf (q) (K * \td^2) (q) dt \\
&\les \int_0^T d_q \td^2_q dt = \int_0^T \l_q \|\Dq v(t) \|_{L^3} \|\Dq \rho (t) \|^2_{L^3} dt.
\end{align*}
Taking $q \to \infty$ of the above, the result then follows from \eqref{main_Theorem_1_2}.
\qed

\section{2$\haf-$dimensional solutions and mixing construction}
\subsection{The construction through perfect mixing}\label{sec:The construction through perfect mixing}
For the proof of Theorem \ref{main_Theorem_1}, we will use $2\frac{1}{2}$-dimensional solutions of the Euler equations
\[
    u(x,t) = (v,\rho),
\]
where $v$ is a two-dimensional divergence-free vector field and $\rho$ is a scalar function satisfying
\[
    \partial_t \rho + v \cdot \nabla \rho = 0.
\]
We apply the smooth mixing construction by Alberti, Crippa, and Mazzucato~\cite{ACM16}, obtaining
\[
    \rho \in C^\infty([0,1)\times \T^2), \qquad v \in C^\infty([0,1)\times \T^2;\R^2),
\]
and such that
\[
    \rho(t) \rightharpoonup 0 \quad \text{weakly in } L^2 \ \text{as } t \to 1^-.
\]
In addition, $v \in L^\infty(0,1; C^\alpha)$ for all $\alpha \in (0,1)$. The next theorem follows from \cite[Section~8]{ACM16}. Recall from \eqref{def:lambdaq} that $\lambda_q = 5^q$.
 
\begin{Theorem} \label{drift_velocity_est}
For each $n \in \N$, there exists a smooth solution of the transport equation $\brho_n \in C^\infty([0,1]\times \T^2)$ with a smooth drift $\bv_n \in C^\infty([0,1]\times \T^2; \R^2)$ such that:
\vspace{0.25cm}
\begin{enumerate}
\item $\|\partial_t^m \bv_n\|_{L^\infty(0,1; C^k)} \le C(k,m)\lambda_n^{k-1}$ for every $k \ge 0$ and $m \in \N$. \label{drift_velocity_est_1}
\vspace{1em}
\item \label{drift_velocity_est_2} $\brho_n(t)$ has zero mean and $\|\brho_n(t)\|_{L^2} = 1$ for all $t \in [0,1]$, and
\[
\|\brho_n(t)\|_{L^\infty} \le 10, \qquad
\|\nabla \brho_n(t)\|_{L^\infty} \le C\lambda_n, \qquad
\|\brho_n(t)\|_{\dot{H}^{-1}(\T^2)} \le C\lambda_n^{-1},
\]
for some absolute constant $C > 0$.
\vspace{1em}
\item $\brho_n(1) = \brho_{n+1}(0)$ for every $n \in \N$.
\end{enumerate}
\end{Theorem}
\noindent We set the smooth initial data for the solutions of both the NSE and the Euler equations to have $\brho_1(0)$ as the third component:
\begin{equation}\label{initial_condition}
u_{\mathrm{in}} := (0,0,\rho_{\mathrm{in}}), \qquad
\rho_{\mathrm{in}} := \brho_1(0) \in C^\infty, \qquad
\|u_{\mathrm{in}}\|_{L^2} = \|\rho_{\mathrm{in}}\|_{L^2} = 1.
\end{equation}
Fix a sequence $\{a_n\}_{n \in \N} \in \Tt$. For each $m \in \N$, we follow the scheme in \cite{BDL23} and glue in time the rescaled velocities $\Bar v_n$ and densities $\Bar\rho_n$, from which we obtain smooth solutions of the transport equation on $[0,1]$:
\begin{align}
v^m(x,t) &=
\sum_{n=0}^m \eta_m'(t)\,
\chi_{[t_n^m,t_{n+1}^m)}(\eta_m(t))
\frac{1}{t_{n+1}^m - t_n^m}
\Bar v_n\!\left(x, \frac{\eta_m(t)-t_n^m}{t_{n+1}^m - t_n^m}\right), \label{glued_velocity}
\\[0.5em]
\rho^m(x,t) &=
\sum_{n=0}^m \chi_{[t_n^m,t_{n+1}^m)}(\eta_m(t))
\Bar\rho_n\!\left(x, \frac{\eta_m(t)-t_n^m}{t_{n+1}^m - t_n^m}\right)
+ \chi_{[t_{m+1}^m,1]}(t)\,\brho_m(x,1), \label{glued_density}
\end{align}
where
\begin{align}\label{def_time_scale}
t_n^m := 1 - (\lambda_n a_n)^{-1} - \tau_m, \qquad 1 \le n \le m+1,
\end{align}
and
\begin{align}
\tau_m &:= m(\nu_m \Lambda_m^2)^{-1}, \label{def_tau_m}\\
\nu_m &:= a_m^2\lambda_m^{-1}, \label{def_nu_m}\\
\Lambda_m &:= a_m^{\haf}\lambda_m. \label{def_Lambda_m}
\end{align}
Observe from the above that $\tau_m = m a_m^{-3}\lambda_m^{-1} \to 0$ as $m \to \infty$, and therefore
\[
t_n^m \to t_n \quad \text{as } m \to \infty,
\]
where $t_n = 1 - (\lambda_n a_n)^{-1}$. The role of $\tau_m$ is to ensure that the dissipation occurs at $t = 1$ by shifting each $t_n^m$ slightly to the left.
\vspace{1em}\\
For each $m > 0$, we set $\eta_m : [0,1] \to [0,1]$ to be a smooth non-decreasing function such that
\begin{enumerate}[label = ({E}\arabic*)]
    \item $\eta_m(t_n^m) = t_n^m$ for every $n \in \N$, \label{def_eta_m1}
    \item $\eta_m^{(k)}(t_n^m) = 0$ for every $n \in \N$ and $k \ge 1$, \label{def_eta_m2}
    \item $|\eta_m^{(k)}(t)|\,\chi_{[t_n^m,t_{n+1}^m)}(\eta_m(t)) \lesssim_k (\lambda_n a_n)^{k-1}$ for every $n \in \N$, $k \ge 2$, and $t \in [0,1]$. \label{def_eta_m3}
\end{enumerate}
From the definition of $\rho^m$, it is straightforward to check that it satisfies the same estimates as in Theorem~\ref{drift_velocity_est},\eqref{drift_velocity_est_2}:
\begin{equation}\label{glued_density_est}
\|\rho^m(t)\|_{L^\infty} \le 10, \qquad
\|\nabla \rho^m(t)\|_{L^\infty} \le C_1 \lambda_m, \qquad
\|\rho^m(t)\|_{\dot{H}^{-1}(\T^2)} \le C_1 \lambda_m^{-1},
\end{equation}
for some absolute constant $C_1 > 0$. Moreover, as $m \to \infty$, there exists a function $\eta(t)$ such that $\eta_m(t) \to \eta(t)$ uniformly, and $\eta(t)$ satisfies \ref{def_eta_m1}--\ref{def_eta_m3} with $t_n^m$ replaced by $t_n$.
\vspace{1em}\\
Next we define
\begin{align}
v(x,t) &= \sum_{n=0}^\infty \eta'(t)\,
\chi_{[t_n,t_{n+1})}(\eta(t))
\frac{1}{t_{n+1}-t_n}\,
\bv_n\Big(x, \frac{\eta(t)-t_n}{t_{n+1}-t_n}\Big), \label{glued_velocity_lim}
\\[0.7em]
\rho(x,t) &= \sum_{n=0}^\infty
\chi_{[t_n,t_{n+1})}(\eta(t))
\brho_n\Big(x, \frac{\eta(t)-t_n}{t_{n+1}-t_n}\Big). \label{glued_density_lim}
\end{align}
\subsection{$2\haf$-dimensional solutions to Euler and Navier--Stokes equations}
One can show that $\rho$ given by \eqref{glued_density_lim} is a smooth solution of the transport equation on $[0,1)$:
\begin{equation}\label{tran_equ_rho}
\partial_t \rho + v \cdot \nabla \rho = 0,
\qquad \rho(0) = \rho_{\mathrm{in}},
\end{equation}
where the drift $v$ is given by \eqref{glued_velocity_lim}. We now introduce the $2\frac{1}{2}$-dimensional velocity field
\begin{equation}\label{def_u}
u(x,t) = \big(v(x_1,x_2,t),\, \rho(x_1,x_2,t)\big).
\end{equation}
One can observe that $u(t)$ is a smooth solution of the Euler equations on $[0,1)$ with
$u(0) = u_{\mathrm{in}}$ defined in \eqref{initial_condition}, and with the force
$f = (g,0)$, where
\begin{equation}\label{def_force_g}
    g := \partial_t v + (v \cdot \nabla)v.
\end{equation}
We consider a sequence of solutions to the NSE that also preserve the $2\frac{1}{2}$-dimensional structure:
\begin{equation}\label{def_NSE_solution_u^m}
    u^m(x,t) = \big(v^m(x_1,x_2,t),\, \theta^m(x_1,x_2,t)\big),
\end{equation}
where $v^m$ is given by \eqref{glued_velocity} and $\theta^m$ satisfies the advection–diffusion equation
\begin{equation}\label{adv_diff_equ}
    \partial_t \theta^m + v^m \cdot \nabla \theta^m = \nu_m \Delta \theta^m,
\end{equation}
with $\theta^m(0) = \rho_{\mathrm{in}}$. One can check that $u^m(t)$ satisfies the NSE with force
\begin{equation}\label{gm_def}
    g^m := \partial_t v^m + (v^m \cdot \nabla)v^m - \nu_m \Delta v^m,
\end{equation}
where $\nu_m$ is given by \eqref{def_nu_m}. The following lemma, taken from \cite{BDL23}, shows that $g^m$ converges strongly to the Euler forcing term.
\begin{Lemma}\label{bounded_force_g}
Let $\nu_m$ and $g^m$ be as above. Then,
\[
    \|g^m\|_{L^1(0,1;L^2(\T^3))} < \infty, \ \  \forall m \in \N.
\]
Moreover, $g^m \to g$ in $L^1([0,1];L^2(\T^3))$ as $m \to \infty$, where
\[
    g = \partial_t v + v \cdot \nabla v
\]
is the body force generated by $v$ in the Euler equations.
\end{Lemma}

\begin{proof}
We start by estimating $\| \partial_t v^m \|_{L^1(0,1;L^2)}$:
\begin{align*}
    \| \partial_t v^m \|_{L^2}
    &\lesssim \sum_{n=0}^m |\eta_m''(t)|
    \chi_{[t_n^m,t_{n+1}^m)}(\eta_m(t))
    \frac{1}{t_{n+1}^m - t_n^m}
    \left\| \bv_n\!\left(\cdot,
    \frac{\eta_m(t) - t_n^m}{t_{n+1}^m - t_n^m}\right)\right\|_{L^2}
    \\
    &\quad + \sum_{n=0}^m |\eta_m'(t)|^2
    \chi_{[t_n^m,t_{n+1}^m)}(\eta_m(t))
    \frac{1}{(t_{n+1}^m - t_n^m)^2}
    \left\| \partial_t \bv_n\!\left(\cdot,
    \frac{\eta_m(t) - t_n^m}{t_{n+1}^m - t_n^m}\right)\right\|_{L^2}.
\end{align*}
From the properties of $\eta_m(t)$, Theorem~\ref{drift_velocity_est}, and the summability of $\{a_n\}$, we obtain
\begin{align*}
    \|\partial_t v^m\|_{L^1(0,1;L^2)}
    &\lesssim \sum_{n=0}^m \lambda_n a_n (\lambda_n a_n)^{-1}
    \lambda_n a_n \lambda_n^{-1} + \sum_{n=0}^m (\lambda_n a_n)^{-1} (\lambda_n a_n)^2
    \lambda_n^{-1}
    \\
    &\lesssim \sum_{n=0}^\infty a_n
    < \infty.
\end{align*}
Thus,
\begin{equation}\label{gm_convergence_dt}
    \lim_{m \to \infty}
    \| \partial_t v^m - \partial_t v \|_{L^1(0,1;L^2)}
    = 0.
\end{equation}
Next, we estimate the nonlinear term $v^m \cdot \nabla v^m$:
\begin{align*}
    \|v^m \cdot \nabla v^m\|_{L^2}
    &= \sum_{n=0}^m |\eta_m'(t)|^2
    \chi_{[t_n^m,t_{n+1}^m)}(\eta_m(t))
    \frac{1}{(t_{n+1}^m - t_n^m)^2}
    \|v_n \cdot \nabla v_n\|_{L^2}
    \\
    \Rightarrow \
    \|v^m \cdot \nabla v^m\|_{L^1(0,1;L^2)}
    &\lesssim \sum_{n=0}^m
    (\lambda_n a_n)^{-1}
    (\lambda_n a_n)^2 \lambda_n^{-1}
    \lesssim \sum_{n=0}^m a_n
    \lesssim 1.
\end{align*}
Therefore,
\begin{equation}\label{gm_convergence_nonlinear}
    \lim_{m \to \infty}
    \| v^m \cdot \nabla v^m - v \cdot \nabla v \|_{L^1(0,1;L^2)}
    = 0.
\end{equation}
Finally, the dissipative term is estimated by
\begin{align*}
    \nu_m \| \Delta v^m(\cdot,t)\|_{L^1(0,1;L^2)}
    &\lesssim \nu_m \sum_{n=0}^m |\eta_m'(t)|
    \chi_{[t_n^m,t_{n+1}^m)}(\eta_m(t))
    \frac{1}{t_{n+1}^m - t_n^m}
    \left\|\Delta \bv_n\!\left(\cdot,
    \frac{\eta_m(t)-t_n^m}{t_{n+1}^m - t_n^m}\right)\right\|_{L^2}
    \\
    &\lesssim \nu_m \sum_{n=0}^m
    (\lambda_n a_n)^{-1} (\lambda_n a_n)\lambda_n
    \lesssim a_m^2,
\end{align*}
and hence this dissipative term tends to $0$ as $m \to \infty$. Together with
\eqref{gm_convergence_dt} and \eqref{gm_convergence_nonlinear}, this concludes the proof.
\end{proof}
\noindent The next lemma provides an estimate for the vanishing viscosity limit of passive scalar equations. The proof can be found in \cite[Proposition~1.3]{DEIJ19}.

\begin{Lemma}\label{lemma:vanish_viscosity_esti}
Let $\nu > 0$, $v \in L^1([0,1]; W^{1,\infty}(\T^2;\R^2))$, and $\theta_0 \in L^2(\T^2)$. Let $\theta^\nu(t)$ be a solution of the advection–diffusion equation
\begin{align*}
    \partial_t \theta^\nu + v \cdot \nabla \theta^\nu &= \nu \Delta \theta^\nu,\\
    \theta^\nu(0) &= \theta_0,
\end{align*}
and $\rho(t)$ be a solution of the transport equation
\begin{align*}
    \partial_t \rho + v \cdot \nabla \rho &= 0,\\
    \rho(0) &= \theta_0.
\end{align*}
Then
\[
    \sup_{t \in [0,1]}
    \|\theta^\nu(\cdot,t) - \rho(\cdot,t)\|_{L^2}^2
    \le \Big(2\nu \int_0^1
    \|\nabla \theta^\nu(\cdot,s)\|_{L^2}^2 ds \Big)^{1/2}
    \Big(2\nu \int_0^1
    \|\nabla \rho(\cdot,s)\|_{L^2}^2 ds \Big)^{1/2}.
\]
\end{Lemma}
\noindent Finally, we state two simple lemmas: the first ensures that the total work done by the limiting force is zero, and the second guarantees that the NSE solutions in our construction converge to a solution of the Euler equations. The original proofs can be found in \cite{cheskidov2023dissipation}.
\begin{Lemma}\label{lemma:zero_work_by_force}
Let $u^\nu = (v^\nu, \theta^\nu)$ be a family of solutions to the 3D NSE on $[0,1]$ with initial data
$u^\nu(0) = u_{\text{in}}$, where
$v^\nu \in C^\infty([0,1]\times \T^3; \R^3)$,
$\theta^\nu \in C^\infty([0,1]\times \T^3)$,
and $f^\nu \to f$ in $L^1(0,1;L^2(\T^3))$.
Assume also that $\{v^\nu\}_\nu$ is bounded in $L^2(0,1;H^1)$ and that
$v^\nu(t) = 0$ for all $\nu$ and all $t \in I \subset [0,1]$.
Then
\begin{equation}\label{energy_loss_theta}
\lim_{\nu \to 0}
\left(
\nu \int_0^t \|\nabla u^\nu\|_{L^2}^2\,d\tau
-
\nu \int_0^t \|\nabla \theta^\nu\|_{L^2}^2\,d\tau
\right)
= 0,
\qquad
\forall t \in [0,1],
\end{equation}
and
\begin{equation}\label{zero_work_by_force}
\int_0^t \langle f, u \rangle\,d\tau = 0,
\qquad
\forall t \in I.
\end{equation}
\end{Lemma}

\begin{proof}
Since $v^\nu \in L^2([0,1];H^1)$, we have
\[
\lim_{\nu \to 0}
\nu \int_0^t \|\nabla v^\nu\|_{L^2}^2\,d\tau
= 0,
\qquad \forall t \in [0,1],
\]
which immediately implies \eqref{energy_loss_theta}.  
From the assumption
\[
v^\nu(t) = 0,
\qquad \forall \nu > 0,~t\in I,
\]
we have
\[
\|u^\nu(t)\|_{L^2} = \|\theta^\nu(t)\|_{L^2},
\qquad \forall\nu,~t\in I.
\]
Using this observation together with the energy equality \eqref{energy_eq_NSE} and
\[
\|\theta^\nu(t)\|_{L^2}^2
= \|u_{\text{in}}\|_{L^2}^2
-2\nu \int_0^t \|\nabla \theta^\nu(\tau)\|_{L^2}^2\,d\tau,
\]
derived from \eqref{adv_diff_equ}, we obtain
\begin{align*}
\int_0^t \langle f, u\rangle\,d\tau
&= \lim_{\nu\to 0} \int_0^t \langle f^\nu, u^\nu\rangle\,d\tau \\
&= \lim_{\nu\to 0}
\bigg(
\nu \int_0^t \|\nabla u^\nu\|_{L^2}^2\,d\tau
-
\nu \int_0^t \|\nabla \theta^\nu\|_{L^2}^2\,d\tau
\bigg)
= 0,
\end{align*}
for any $t \in I$.
\end{proof}

\begin{Lemma}\label{lemma:theta_convergence}
Let $\{\theta^\nu\}_{\nu > 0}$ be a family of solutions to the advection–diffusion equation
\begin{align*}
    \partial_t \theta^\nu + v^\nu \cdot \nabla \theta^\nu &= \nu \Delta \theta^\nu,\\
    \theta^\nu(0) &= \theta_0,
\end{align*}
bounded in $L^\infty([0,T];L^2(\T^3))$,
with $\theta_0 \in L^2(\T^3)$ and
$v^\nu \to v$ in $L^1([0,T];L^2(\T^3))$ as $\nu\to 0$.
Then there exists a subsequence $\nu_j \to 0$ such that
\[
\theta^{\nu_j} \to \theta \quad \text{in} \quad C_w([0,T];L^2),
\]
where $\theta$ is a weak solution of the transport equation with drift $v$:
\begin{align*}
    \partial_t \theta + v \cdot \nabla \theta &= 0,\\
    \theta(0) &= \theta_0.
\end{align*}
\end{Lemma}

\begin{proof}
    Consider the Fourier expansion of $\th^\nu(t)$, 
\[
\th^\nu (x,t) = \sum_{k \in \Z} \hat{\th}_k^\nu (t) e^{-ik\cdot x}
\]
for each $k \in \Z$, we have that for any $t \in [0,T]$,
\begin{enumerate}[label=\itshape(\roman*)] 
    \item \label{uni_bdded_nu} $|\hat\th^\nu_k (t)| < M$ for some positive $M$, uniformly in $\nu$, \\
    \item \label{uni_equicont_nu} $\hat \th^\nu_k (t)$ is uniformly equicontinuous in $\nu$, \\
\end{enumerate}
where \ref{uni_bdded_nu} follows from the energy inequality
    \[
      \twonorm{\th^\nu (t)} \leq \twonorm{\th_0}
    \]
and that $|\hat\th^\nu_k (t)| \leq \|\hat\th^\nu_k (t)\|_{l^2} = \twonorm{\th^\nu (t)}$. In order to establish \ref{uni_equicont_nu}, we look at the weak formulation of the equation of the Fourier coefficients
\begin{align*}
\hat \th^\nu_k (t) - \hat \th^\nu_k (s) &= -\int_s^t  \hat v^\nu * k \hat \th^\nu_k (\tau) d\tau - \int_s^t \nu k^2 \hat \th^\nu_k(\tau) d\tau \\
& = -\int_s^t  \sum_{n \in \Z} \hat v^\nu_{k-n} (\tau) \hat \th_n (\tau) d\tau - \int_s^t \nu k^2 \hat \th^\nu_k(\tau) d\tau
\end{align*}
for $0 \leq s<t \leq T$. Therefore,
\begin{align*}
    |\hat \th^\nu_k (t) - \hat \th^\nu_k (s)| &\leq \bigg ( \int_s^t \| \hat v_k (\tau) \|_{l^2} \| \hat \th^\nu_k (\tau) \|_{l^2} d\tau + \nu k^2 \int_s^t  | \hat \th^\nu_k (\tau) | d\tau \bigg)\\
    &\leq \bigg( \| \th^\nu (\cdot) \|_{L^\infty([0,T],L^2)} \|v^\nu(\cdot) \|_{L^1([0.T],L^2)} + \nu k^2 M \bigg) |t-s|.
\end{align*}
These together with the assumption on $\th^\nu$ and $v^\nu$ establishes the uniform equicontinuity in $\nu$, for fixed $k \in \Z$. \\\\
By Arzela-Ascoli Theorem, there exists a continuous in time function $\hat \th_k(t)$ such that 
\[
\hat \th^{\nu_j}_k(t) \to \hat \th_k(t) \ \ \ \text{in} \ \ C([0,1],l^2)
\]
for a subsequence $\nu_j \to 0$ as $j \to \infty$. Now with $\hat \th^{\nu_j}_{k+1}(t)$ we extract a further subsequence by a similar manner that converges to $\hat \th_{k+1}(t)$ in $C([0,1],l^2)$. With a diagonal argument one can obtain a subsequence (still denoted as) $\nu_{j}$ such that \[
\th^{\nu_{j}}(t) \to \th(t) \ \ \ \text{in} \ \ C([0,1],L_w^2) \ \ \text{as} \ \ \nu_j \to 0,
\]
where 
\[
\th(t) = \sum_{k \in \Z} \hat \th_k (t) e^{-ik\cdot x}.
\]
To see the convergent is in $C([0,1],L_w^2)$, one can apply a density argument and fix a function $\phi \in L^2$ and such that $\hat \phi_k \equiv 0$ for $|k| > N_0$ for some $N_0 \in \N$. Observe that 
\begin{align*}
  \<\th(t), \phi \> &:= \sum_{k \in \Z} \hat \th_k (t) \hat \phi_k = \sum_{|k| \leq N_0} \hat \th_k (t) \hat \phi_k  \\
  &= \lim_{j \to \infty } \sum_{|k| \leq N_0} \hat \th^{\nu_j} (t) \hat \phi_k \\
  &= \lim_{j \to \infty} \<\th^{\nu_j}(t), \phi \>,
\end{align*}
uniformly in time. This proves the existence of a weak convergent sequence. The fact that $\th(t)$ is a weak solution to the transport equation follows from 
\[
v^{\nu_j} \th^{\nu_j} - v\th = \big(v^{\nu_j} -v \big)\th^{\nu_j} + v \big(\th^{\nu_j} -\th \big),
\]
which goes to $0$ weakly.
\end{proof}

\section{Proof of Theorem \ref{main_Theorem_2}}\label{proof of main_Theorem_2}
\noindent As noted in \eqref{tran_equ_rho}–\eqref{def_force_g}, $\rho(t)$ is a weak solution of the transport equation
\[
\partial_t \rho + v \cdot \nabla \rho = 0
\]
and $u(t) = (v(t),\rho(t))$ is a weak solution of the forced Euler equations with
\[
f = (g,0) = (\partial_t v + v \cdot \nabla v,\,0)
\]
and initial condition $u(0) = u_{\text{in}} = (0,\rho_{\text{in}})$ on $[0,1]$ (indeed a smooth solution on $[0,1)$). Furthermore, $\rho^m(t)$ defined in \eqref{glued_density} satisfies the transport equation with drift $v^m(t)$ from \eqref{glued_velocity} on $[0,1]$:
\begin{align*}
\partial_t \rho^m + v^m \cdot \nabla \rho^m &= 0, \\
\rho^m(0) &= \rho_{\text{in}}.
\end{align*}
Let $\theta^m$ be the unique smooth solution of the advection–diffusion equation
\begin{equation}\label{adv_equ_thetam}
    \partial_t \theta^m + v^m \cdot \nabla \theta^m = \nu_m \Delta \theta^m
\end{equation}
on $[0,1]$ with $\theta^m(0) = \rho_{\text{in}}$. Then the corresponding family of smooth NSE solutions on $[0,1]$ is
\begin{equation*}
u^m(t) = \big(v^m(t), \theta^m(t)\big),
\end{equation*}
with forcing $f^m = (g^m,0)$, where $g^m$ is given by \eqref{gm_def}. Since $v^m(0)=0$ for all $m$, we have $u^m(0) = u_{\text{in}} = (0,\rho_{\text{in}})$.
\vspace{1em}\\
We now prove that
\begin{equation} \label{proof of main_Theorem_1_1}
    u \in L^3\big([0,1]; \Bb^{\frac{1}{3},a_n}_{3,\infty}(\T^3)\big),
\end{equation}
where $\{a_n\}_{n \in \N} \in \Tt$ is the fixed sequence from the construction in Section~\ref{sec:The construction through perfect mixing}. Since $u(t) = (v(t),\rho(t))$, we estimate $\|v(t)\|_{L^3_t\Bb^{\frac{1}{3},a_n}_{3,\infty}}$ and $\|\rho(t)\|_{L^3_t\Bb^{\frac{1}{3},a_n}_{3,\infty}}$ separately.
\vspace{2em}\\
For the velocity, by Theorem~\ref{drift_velocity_est}, \eqref{drift_velocity_est_1},
\begin{align*}
    \|v\|_{L^3_t\Bb^{\frac{1}{3},a_n}_{3,\infty}}^3
    &= \bigg\|
    \sum_{n=0}^\infty \eta'(\cdot)\,\chi_{[t_n,t_{n+1}]}(\eta(\cdot))\,\frac{1}{t_{n+1}-t_n}\,
    \bar v_n\!\left(\cdot,\frac{\eta(\cdot)-t_n}{t_{n+1}-t_n}\right)
    \bigg\|_{L^3_t\Bb^{\frac{1}{3},a_n}_{3,\infty}}^3\\
    &\lesssim \sum_{n=0}^\infty
    \bigg\|
    \eta'(\cdot)\,\chi_{[t_n,t_{n+1}]}(\eta(\cdot))\,\frac{1}{t_{n+1}-t_n}\,
    \bar v_n\!\left(\cdot,\frac{\eta(\cdot)-t_n}{t_{n+1}-t_n}\right)
    \bigg\|_{L^3_t\Bb^{\frac{1}{3},a_n}_{3,\infty}}^3 \\
    &\lesssim \sum_{n=0}^\infty (\lambda_n a_n)^{-1} (\lambda_n a_n)^3 \|\bar v_n\|_{\Bb^{\frac{1}{3},a_n}_{3,\infty}(\T^3)}^3\\
    &\lesssim \sum_{n=0}^\infty  (\lambda_n a_n)^2
    \bigg(\sum_{q \ge -1} a_{q+2}\lambda_q^{1/3}\|\Delta_q \bar v_n\|_{L^3}\bigg)^3 \\
    &\lesssim \sum_{n=0}^\infty (\lambda_n a_n)^2
    \bigg(\sum_{q \ge -1} a_{q+2}\,\|\nabla^{1/3}\Delta_q \bar v_n\|_{L^3}\bigg)^3 \\
    &\lesssim \sum_{n=0}^\infty (\lambda_n a_n)^2 \lambda_n^{-2}
    \bigg(\sum_{q=1}^\infty a_q \bigg)^3
    \lesssim \sum_{n=0}^\infty a_n^2 < \infty,
\end{align*}
where we used Bernstein’s inequality (Lemma~\ref{lemma:Bernstein}). The estimate for $\rho$ is analogous, using Theorem~\ref{drift_velocity_est}, \eqref{drift_velocity_est_2}:
\begin{align*}
    \|\rho\|_{L^3_t\Bb^{\frac{1}{3},a_n}_{3,\infty}}^3
    &= \bigg\|
    \sum_{n=0}^\infty \eta'(\cdot)\,\chi_{[t_n,t_{n+1}]}(\eta(\cdot))\,
    \bar\rho_n\!\left(\cdot,\frac{\eta(\cdot)-t_n}{t_{n+1}-t_n}\right)
    \bigg\|_{L^3_t\Bb^{\frac{1}{3},a_n}_{3,\infty}}^3\\
    &\lesssim \sum_{n=0}^\infty
    \bigg\| \eta'(\cdot)\,\chi_{[t_n,t_{n+1}]}(\eta(\cdot))\,
    \bar\rho_n\!\left(\cdot,\frac{\eta(\cdot)-t_n}{t_{n+1}-t_n}\right)
    \bigg\|_{L^3_t\Bb^{\frac{1}{3},a_n}_{3,\infty}}^3 \\
    &\lesssim \sum_{n=0}^\infty (\lambda_n a_n)^{-1}
    \|\bar\rho_n\|_{\Bb^{\frac{1}{3},a_n}_{3,\infty}(\T^3)}^3\\
    &\lesssim \sum_{n=0}^\infty (\lambda_n a_n)^{-1}
    \bigg( \sum_{-1 \le q \le n} a_{q+2}\lambda_q^{1/3}\|\Delta_q \bar\rho_n\|_{L^3} \bigg)^3
    + \sum_{n=0}^\infty (\lambda_n a_n)^{-1}
    \bigg( \sum_{q>n} a_{q+2}\lambda_q^{1/3}\|\Delta_q \bar\rho_n\|_{L^3} \bigg)^3 \\
    &\lesssim \sum_{n=0}^\infty a_n^2
    + \sum_{n=0}^\infty (\lambda_n a_n)^{-1}
    \bigg( \sum_{q>n} a_{q+2}\lambda_q^{-2/3}\|\Delta_q \nabla \bar\rho_n\|_{L^3} \bigg)^3 \\
    &\lesssim \sum_{n=0}^\infty a_n^2
    + \sum_{n=0}^\infty (\lambda_n a_n)^{-1}
    \bigg( \sum_{q>n} a_{q+2}\lambda_q^{-2/3}\lambda_n \bigg)^3 \\
    &\lesssim \sum_{n=0}^\infty a_n^2 < \infty.
\end{align*}
This established \eqref{proof of main_Theorem_1_1}.
\vspace{1em}\\
The convergence $u^m \to u$ follows from the definitions \eqref{glued_velocity_lim}–\eqref{glued_density_lim} and Lemma~\ref{lemma:theta_convergence} (after passing to a subsequence if necessary). The convergence $f^m \to f$ follows from Lemma~\ref{bounded_force_g}.
\vspace{1em}\\
To prove \eqref{main_Theorem_2_1}, we use \eqref{glued_velocity_lim}–\eqref{glued_density_lim} together with Theorem~\ref{drift_velocity_est} to obtain
\begin{align*}
 \int_0^1 \lambda_q \,\|\Delta_q v(t)\|_{L^3}\,\|\Delta_q \rho(t)\|_{L^3}^2 \,dt
 &\le \bigg( \int_0^1 \lambda_q \|\Delta_q v\|_{L^3}^3 dt \bigg)^{\!1/3}
       \bigg( \int_0^1 \lambda_q \|\Delta_q \rho\|_{L^3}^3 dt \bigg)^{\!2/3} \\
&\le \big(\lambda_q \lambda_q^{-3} (\lambda_q a_q)^3 (\lambda_q a_q)^{-1}\big)^{1/3}
      \big(\lambda_q (\lambda_q a_q)^{-1}\big)^{2/3} \\
&\le a_q^{2/3} a_q^{-2/3} \lesssim 1.
\end{align*}
Taking $\limsup_{q\to\infty}$ on both sides yields the claim. A similar argument yields \eqref{main_Theorem_2_2}.
 \\\\
We now show that with the choice of $\nu_m$ in \eqref{def_nu_m}, we obtain total viscous anomalous dissipation at time $t = 1$, namely \eqref{main_Theorem_2_3}. Let $P_{\le \Lambda_m}$ denote the frequency projection below $\Lambda_m$ defined in \eqref{def_Lambda_m}. Recall the energy equality for $\theta^m$:
\begin{equation}\label{energy_equ_thetam}
    \|\theta^m(t)\|_{L^2}^2
    + 2\nu_m \int_0^t \|\nabla \theta^m(\tau)\|_{L^2}^2\, d\tau
    = \|\theta^m(0)\|_{L^2}^2.
\end{equation}
Thus,
\[
    2\nu_m \int_0^1 \|\nabla \theta^m(t)\|_{L^2}^2\,dt
    \le \|\theta^m(0)\|_{L^2}^2 = 1.
\]
By Lemma~\ref{lemma:vanish_viscosity_esti}, for any $T \in [t_{m+1}^m,1]$,
\begin{align}
    \sup_{t\in [0,t_{m+1}^m]} \|\theta^m(t)-\rho^m(t)\|_{L^2}^2
    &\le
    \left(
    2\nu_m \int_0^T \|\nabla \rho^m\|_{L^2}^2 dt
    \right)^{1/2}
    \left(
    2\nu_m \int_0^T \|\nabla \theta^m\|_{L^2}^2 dt
    \right)^{1/2} \notag\\
    &\le
    \left(
    2\nu_m \int_0^T \|\nabla \rho^m\|_{L^2}^2 dt
    \right)^{1/2} \notag\\
    &\lesssim \sqrt{\nu_m (T - t_m^m)}\,\lambda_m.
    \label{L2_diff_est}
\end{align}
From \eqref{def_time_scale}, we have
\[
t_{m+1}^m - t_m^m \sim (\lambda_m a_m)^{-1},
\]
and $\nu_m = a_m^2 \lambda_m^{-1}$ gives
\vspace{1em}\\
\[
    \sup_{t\in [0,t_{m+1}^m]}
    \|\theta^m(t)-\rho^m(t)\|_{L^2}^2
    \lesssim a_m^{1/2}.
\]
From \eqref{glued_density_est},
\begin{align}
    \|P_{\le \Lambda_m} \rho^m(t)\|_{L^2}
    &\le \Lambda_m \|\rho^m(t)\|_{\dot{H}^{-1}}
    \notag\\
    &\lesssim a_m^{1/2},
    \label{L2_lowmode_est_rho}
\end{align}
for $t \in [t_m^m, t_{m+1}^m]$. Hence,
\begin{align*}
    \|P_{\le \Lambda_m} \theta^m(t_{m+1}^m)\|_{L^2}
    &\le
    \|P_{\le \Lambda_m}(\theta^m - \rho^m)(t_{m+1}^m)\|_{L^2}
    + \|P_{\le \Lambda_m}\rho^m(t_{m+1}^m)\|_{L^2} \\
    &\lesssim a_m^{1/2}
    \xrightarrow[m\to\infty]{} 0.
\end{align*}
\noindent Since $v^m \equiv 0$ on $[t_{m+1}^m,1]$, $\theta^m$ solves the heat equation
\[
    \partial_t \theta^m = \nu_m \Delta \theta^m,
    \quad t \in [t_{m+1}^m,1],
\]
and the energy inequality yields
\begin{align}
    \|P_{\le \Lambda_m} \theta^m(1)\|_{L^2}^2
    &\le
    \|P_{\le \Lambda_m} \theta^m(t_{m+1}^m)\|_{L^2}^2 \notag\\
    &\lesssim a_m
    \xrightarrow[m\to\infty]{} 0.
    \label{conv_low_mode1}
\end{align}
We now turn to the high-frequency component. Again, using the fact that $\theta^m$ satisfies the heat equation for $t \in [t_{m+1}^m,1]$, we compute
\begin{align*}
    \frac{d}{dt} \|P_{> \Lambda_m} \theta^m(t)\|_{L^2}^2
    &= -2\nu_m \|\nabla P_{> \Lambda_m} \theta^m(t)\|_{L^2}^2 \\
    &\le -2\nu_m \Lambda_m^2 \|P_{> \Lambda_m} \theta^m(t)\|_{L^2}^2.
\end{align*}
Therefore, for every $t \in [t_{m+1}^m, 1]$,
\begin{align}
    \|P_{> \Lambda_m}\theta^m(t)\|_{L^2}^2
    &\le \|\theta^m(t_{m+1}^m)\|_{L^2}^2
    e^{-2\nu_m \Lambda_m^2 (t - t_{m+1}^m)} \notag \\
    &\le \|\theta^m(0)\|_{L^2}^2
    e^{-2\nu_m \Lambda_m^2 (t - t_{m+1}^m)} \notag \\
    &\le e^{-2\nu_m \Lambda_m^2 (t - t_{m+1}^m)}. \notag
\end{align}
Hence,
\begin{equation}\label{conv_high_mode1}
    \|P_{> \Lambda_m}\theta^m(1)\|_{L^2}^2
    \le e^{-2\nu_m \Lambda_m^2 (1 - t_{m+1}^m)}
    \le e^{-2m}
    \xrightarrow[m\to\infty]{} 0.
\end{equation}
The final inequality uses the bound $1 - t_{m+1}^m \ge \tau_m$ and
\[
\nu_m \Lambda_m^2 (1 - t_{m+1}^m)
\ge \nu_m \Lambda_m^2 \tau_m
= m.
\]
Combining \eqref{conv_low_mode1} and \eqref{conv_high_mode1}, we conclude that
\begin{equation}\label{conv_theta^m}
    \lim_{m \to \infty}
    \|\theta^m(1)\|_{L^2}^2 = 0.
\end{equation}
Finally, using \eqref{conv_theta^m} together with the energy equality
\eqref{energy_equ_thetam},
\begin{equation}\label{TDA_conclusion}
    \lim_{m \to \infty} 2\nu_m \int_0^1 \|\nabla \theta^m(t)\|_{L^2}^2\,dt
    = \|\theta^m(0)\|_{L^2}^2
    - \lim_{m \to \infty} \|\theta^m(1)\|_{L^2}^2
    = 1
    = \|u_{\mathrm{in}}\|_{L^2}^2.
\end{equation}
Thus, by Lemma~\ref{lemma:zero_work_by_force}, we obtain \eqref{main_Theorem_2_3}.
This completes the proof of Theorem~\ref{main_Theorem_2}.
\qed

\section{Second construction}\label{Sec_Second_construction}
\noindent In this section, we set $\tilde{\lambda}_n := 2^n$. For the proof of Theorem~\ref{main_Theorem_3}, we construct a weak solution to the forced Euler equations where the external force is slightly rougher such that the sequence $f^{\nu_n}$ may not converge in $L^1([0,T];L^2(\T^3))$. To that end, we begin with the following definition.
\begin{Definition}[\cite{CL21}]\label{def:N(Q_T)}
Let $T > 0$ and define the spacetime domain $Q_T := \T^3 \times \left[0,T \right)$. We say that a pair $\left( u,f \right)$ is a smooth solution of the Euler equations~\eqref{eq:Euler} with finite energy input if the following conditions hold:
\begin{enumerate}[label=\textit{(\roman*)}]
\item $u \in C^\infty_{t,x}(Q_T) \cap L^\infty([0,T];L^2(\T^3))$ and $f \in C^\infty_{t,x}(Q_T)$. \label{def:N(Q_T)_1}
\item The pair $(u,f)$ satisfies the Euler equations~\eqref{eq:Euler} on $Q_T$. \label{def:N(Q_T)_2}
\item The duality pairing $\langle f,u\rangle \in L^1_{\mathrm{loc}}([0,T))$ and the limit
\[
\lim_{t \to T^-} \int_0^t \langle f(s),u(s) \rangle \, ds
\]
exists. \label{def:N(Q_T)_3}
\item The limit $\lim_{t \to T^-} f(t)$ exists in $\mathcal{D}'(\T^3)$. \label{def:N(Q_T)_4}
\end{enumerate}
\end{Definition}
\noindent We will use the shorthand notation $(u,f) \in \mathcal{N}(Q_T)$ whenever $(u,f)$ satisfies Definition~\ref{def:N(Q_T)}.  
Assuming $(u,f) \in \mathcal{N}(Q_T)$, one may derive the following truncated energy equality by multiplying the Euler equations by $\Ss_q \Ss_q u$ and integrating over $\T^3 \times (T-h,T)$:
\begin{equation}\label{eq:Energy_equality_Q_T}
\frac{1}{2}\| \Ss_q u(T) \|_{L^2}^2
-
\frac{1}{2}\| \Ss_q u(T-h) \|_{L^2}^2
=
- \int_{T-h}^{T} \langle \div (u \otimes u), \Ss_q \Ss_q u \rangle \, dt
+
\int_{T-h}^{T} \langle \Ss_q f, \Ss_q u \rangle \, dt.
\end{equation}

\begin{Definition}[\cite{CL21}]\label{def:Anomalous_dissipation_q_shell}
For $(u,f) \in \mathcal{N}(Q_T)$, we define the inviscid anomalous dissipation at the $q$-th dyadic shell by
\begin{equation}\label{def:anomalous_dissipation_amount}
\Pi_q(t,h)
:=
\int_{[t-h,t+h]\cap [0,T]}
\left\langle \div(u \otimes u), \Ss_q(\Ss_q u) \right\rangle \, d\tau,
\end{equation}
and the inviscid anomalous work at the $q$-th dyadic shell by
\begin{align}\label{def:anomalous_work_q_shell_1}
\Phi_q(t,h)
&:=
\int_{[t-h,t+h]\cap [0,T]}
\big( \langle \Ss_q f, \Ss_q u \rangle - \langle f,u\rangle \big) \, d\tau,
\ \ \ t < T, \\
\Phi_q(T,h)
&:=
\int_{T-h}^{T} \langle \Ss_q f, \Ss_q u \rangle \, d\tau
-
\lim_{t_1 \to T^-}
\int_{T-h}^{t_1} \langle f,u\rangle \, d\tau.
\label{def:anomalous_work_q_shell_2}
\end{align}
\end{Definition}
\noindent With Definition~\ref{def:anomalous_dissipation_amount}, the identity
\eqref{eq:Energy_equality_Q_T} can be rewritten as
\[
\frac{1}{2}\| \Ss_q u(T)\|_{L^2}^2 - \frac{1}{2} \|\Ss_q u(T-h)\|_{L^2}^2
=
\lim_{t_1 \to T^-} \int_{T-h}^{t_1} \langle f,u \rangle \, d\tau
-
\Pi_q(T,h)
+
\Phi_q(T,h).
\]
Since $u \in L^\infty([0,T];L^2)$ by
Definition~\ref{def:N(Q_T)}\ref{def:N(Q_T)_1},
taking $\limsup_{q \to \infty}$ on both sides yields the following
energy jump relation:
\begin{equation}\label{eq:energy_jump_1}
\frac{1}{2}\|u(T)\|_{L^2}^2 - \frac{1}{2}\|u(T-h)\|_{L^2}^2
=
\lim_{t_1 \to T^-} \int_{T-h}^{t_1} \langle f,u \rangle \, d\tau
+
\limsup_{q \to \infty}\big( -\Pi_q(T,h) + \Phi_q(T,h) \big).
\end{equation}\\
If, in addition, $\langle f,u \rangle \in L^1(0,T)$ and
$\langle f,u\rangle \ge 0$ on $(0,T)$, then
\eqref{eq:energy_jump_1} simplifies to
\begin{equation}\label{eq:energy_jump_2}
\frac{1}{2}\|u(T)\|_{L^2}^2 - \frac{1}{2}\|u(T-h)\|_{L^2}^2
=
\int_{T-h}^{T} \langle f,u \rangle \, d\tau
+
\limsup_{q \to \infty}\big( -\Pi_q(T,h) + \Phi_q(T,h) \big),
\end{equation}
where
\[
\Phi_q(T,h)
=
\int_{T-h}^{T} \big( \langle \Ss_q f, \Ss_q u \rangle
- \langle f,u \rangle \big)\, d\tau.
\]
\begin{Definition}[\cite{CL21}]
Let $(u,f) \in \mathcal{N}(Q_T)$. For any $t \in [0,T]$, the
\emph{amount of invisicid anomalous dissipation} is defined by
\begin{equation}\label{def:AD_amount}
\Pi(t) := \limsup_{q \to \infty} |\Pi_q(t,h)|
\end{equation}
and the \emph{amount of anomalous work} is defined by
\begin{equation}\label{def:AW_amount}
\Phi(t) := \limsup_{q \to \infty} |\Phi_q(t,h)|.
\end{equation}
\end{Definition}
\noindent We next introduce a family of building blocks used in the second construction.
\begin{Lemma}[\cite{CL21}]\label{lemma:build_block_second_construction}
Let $d=3$ and $0 \le \beta < d$. Then there exist $N \in \N$ and a family of divergence-free vector fields $\{W_n\}_{n \ge N} \subset C^\infty(\T^d)$ such that:
\begin{enumerate}[label=\textit{(\roman*)}]
\item\label{lemma:build_block_second_construction_1}
The Fourier support of $W_n$ is contained in an annulus of radius $\tilde{\lambda}_n$:
\[
\operatorname{supp}\widehat{W}_n \subset \{k : \tilde{\lambda}_n \le |k| \le \tilde{\lambda}_{n+1} \}.
\]

\item\label{lemma:build_block_second_construction_2}
Their $L^r$–norms satisfy
\[
\|W_n\|_{L^2} = 1
\quad \text{and} \quad
\|W_n\|_{L^r} \sim_r \tilde{\lambda}_n^{\beta(\frac{1}{2}-\frac{1}{r})}.
\]

\item\label{lemma:build_block_second_construction_3}
The energy flux through dyadic shells satisfies
\[
\lim_{n\to\infty}
\tilde{\lambda}_n^{-1-\frac{\beta}{2}}
\int_{\T^3}
\div\big( \Ss_n(W_n\otimes W_n)\big)\cdot (\Ss_n W_n)\,dx
= 2,
\]
and for $q\ne n$,
\[
\int_{\T^3}
\div\big( \Ss_q(W_n\otimes W_n)\big)\cdot (\Ss_q W_n)\,dx =0.
\]
\end{enumerate}
\end{Lemma}

\begin{Remark}
The parameter $\beta$ is related to the intermittency dimension $D$ via
\(
D = 3 - \beta.
\)
See, for example, \cite{CL20} or \cite{cheskidov2024optimal} for a definition of intermittency dimension using Bernstein inequalities.
\end{Remark}
\vspace{1em}
\noindent Let $N \in \N$ and the family $\{W_n\}_{n\ge N}$ be given by Lemma~\ref{lemma:build_block_second_construction} with a fixed $\beta<3$, to be specified later. Define the time scale
\begin{equation}\label{def:second_construction_time_step}
\tau_n
:=
\frac{1}{4}\big(1 - 2^{-1-\frac{\beta}{2}}\big)^{-1}
\tilde{\lambda}_n^{-1-\frac{\beta}{2}},
\qquad n \ge N.
\end{equation}
Moreover, for every $\varepsilon>0$ there exists a constant $c_\varepsilon>0$ depending only on $\varepsilon$ such that
\begin{equation}\label{bound:time_step_condition}
\tau_n > \tau_{n+1} + c_\varepsilon\, \tau_{n+1}^{1+\frac{\varepsilon}{4}},
\qquad \text{for all } n \ge N.
\end{equation}
Let $h \in C^\infty(\R)$ be a smooth cutoff function such that
\[
h(t)=
\begin{cases}
0, & t\le -1,\\
1, & t\ge 0,
\end{cases}
\qquad
h^{1/2},\, (1-h)^{1/2}\in C^\infty(\R).
\]
Define
\[
h_n(t)
:=
h\!\left(\frac{t+1}{c_\varepsilon \tau_n^{\varepsilon/4}}\right),
\]
so that $h_n \in C^\infty(\R)$ and
\[
h_n(t)=
\begin{cases}
0,& t \le -1 - c_\varepsilon \tau_n^{\varepsilon/4},\\
1,& t \ge -1.
\end{cases}
\]
Furthermore,
\[
|h_n'(t)| \lesssim c'_\varepsilon \tau_n^{-\varepsilon/4}.
\]
Now define
\begin{equation}\label{def:time_gluing_second_construction}
\chi_n(t)
:=
\Bigg[
h_n\!\left( \frac{t-T}{\tau_n} \right)
-
h_{n+1}\!\left( \frac{t-T}{\tau_{n+1}} \right)
\Bigg]^{1/2},
\qquad T := \tau_N.
\end{equation}
Using \eqref{bound:time_step_condition}, one checks that \(\chi_n\) has the piecewise structure:
\[
\chi_n(t)=
\begin{cases}
0,
& t \le T-\tau_n - c_\varepsilon \tau_n^{1+\varepsilon/4},\\[4pt]
h_n^{1/2}\!\left(\frac{t-T}{\tau_n}\right),
& T-\tau_n - c_\varepsilon \tau_n^{1+\varepsilon/4}
\le t \le
T-\tau_n,\\[4pt]
1,
& T-\tau_n \le t \le T-\tau_{n+1}-c_\varepsilon \tau_{n+1}^{1+\varepsilon/4},\\[4pt]
\big(1-h_{n+1}((t-T)/\tau_{n+1})\big)^{1/2},
& T-\tau_{n+1}-c_\varepsilon \tau_{n+1}^{1+\varepsilon/4}
\le t \le T-\tau_{n+1},\\[4pt]
0,
& t \ge T-\tau_{n+1}.
\end{cases}
\]
Hence, for all \(t \in [0,T)\),
\begin{equation}\label{eq:sum_char_n}
\sum_{n=N}^\infty \chi_n(t)^2 = 1 
\ \  \text{and} \ \ 
\sum_{n=N}^\infty \chi_n(T)^2 = 0.
\end{equation}
Since each \(h_n\) is a translation and rescaling of \(h\), it follows that \(\chi_n \in C_c^\infty(\R)\). Finally, we have the derivative bound
\begin{equation}\label{eq:chi_derivative_est}
|\chi_n'(t)|
\lesssim
\tau_n^{-1-\varepsilon/4}\,
\mathbbm{1}_{[T-\tau_n-c_\varepsilon \tau_n^{1+\varepsilon/4},\,T-\tau_n]}
+
\tau_{n+1}^{-1-\varepsilon/4}\,
\mathbbm{1}_{[T-\tau_{n+1}-c_\varepsilon \tau_{n+1}^{1+\varepsilon/4},\,T-\tau_{n+1}]}.
\end{equation}
We are now ready to introduce the solutions to the forced Euler and
Navier–Stokes equations. First, define the Euler solution
\begin{equation}\label{def:second_construction_Euler}
u(x,t)
=
\sum_{n=N}^{\infty} \chi_n(t)\, W_n(x),
\end{equation}
with the external force
\begin{equation}\label{def:second_construction_Euler_force}
f(x,t)
=
\partial_t u + \operatorname{div}(u \otimes u).
\end{equation}

\begin{Lemma}\label{lemma:forced_Euler_properties}
The pair $(u,f)$ defined above satisfies: \\
\begin{enumerate}[label=\textit{(\roman*)}]
\item $(u,f) \in \mathcal{N}(Q_T)$; \label{lemma:forced_Euler_properties_1}\\

\item $\|u(t)\|_{L^2}=1$ for all $t \in [0,T)$, while $\|u(T)\|_{L^2}=0$; \label{lemma:forced_Euler_properties_2}\\

\item $u \in L^\infty([0,T];L^2(\T^3))$ and
\[
\lim_{t \to T^-}\langle u(t),\phi\rangle = 0
\quad
\text{for all } \phi \in L^2(\T^3);\label{lemma:forced_Euler_properties_3}
\]
\item $f \in L^{1,w}([0,T];L^2(\T^3))$. \label{lemma:forced_Euler_properties_4}
\end{enumerate}
\end{Lemma}
\vspace{1em}
\begin{proof}
The conditions \ref{def:N(Q_T)_1} and \ref{def:N(Q_T)_2} in definition $\ref{def:N(Q_T)}$ follow directly from \eqref{def:second_construction_Euler} and \eqref{def:second_construction_Euler_force}. Now since $(u,f)$ is a smooth solution on $Q_T$, they satisfy the energy equality
\[
\twonorm{u(t)}^2 - \twonorm{u_\initial}^2 = \int_0^t \<f(\tau),u(\tau) \> d\tau, \ \ \ 0 \leq t < T.
\]
From \eqref{eq:sum_char_n} we have $\lim_{t \to T^-} \twonorm{u(t)}^2 =\twonorm{u_\initial}^2 = 1$, hence by sending $t \to T^-$ of the above:
\[
\lim_{t \to T^-} \int_0^t \<f,u \> d\tau = 0,
\]
in particular the limit exists. Finally, the fact that the $\lim_{t \to T^-} f(t)$ exists in distribution follows from the Plancherel's identity. Therefore $(u,f) \in \Nn(Q_T)$. Property \ref{lemma:forced_Euler_properties_2} follows immediately from the definition of $u(t)$ and $\sum_{n \geq N} \chi^2_n(T) = 0$, from which we see also that $u \in L^\infty ([0,T];L^2(\T^3))$. By lemma \ref{lemma:build_block_second_construction}, \ref{lemma:build_block_second_construction_2}, we conclude that 
\[
\lim_{t \to T^-} \<u(t), \phi \> = 0, \ \ \ \forall \phi \in L^2(\T^3)
\]
using the $L^r(\T^3)$ scaling with $r < 2$. To show \ref{lemma:forced_Euler_properties_4}, first recall that 
\[
f = \p_t u + \divv (u \otimes u).
\]
Then notice that 
\[
\twonorm{\p_t u} \les \sum_{n = N}^\infty |\chi'_n(t)|
\]
and so by the estimate \eqref{eq:chi_derivative_est} we have for $n \geq N$:
\[
\tau^{-1-\frac{\e}{4}}_n \mu \big( t \in [0,T]: \twonorm{\p_t u} > \tau_n^{-1-\frac{\e}{4}} \big) \les \tau^{-1-\frac{\e}{4}}_n \tau^{1+\frac{\e}{4}}_{n+1} \les 1.
\]
For the inertia term, using lemma \ref{lemma:build_block_second_construction}, \ref{lemma:build_block_second_construction_2}, we obtain
\begin{align*}
\twonorm{\divv (u\otimes u)} &\les \sum_{n =N}^\infty \chi^2_n(t) \twonorm{\nabla W_n} \Linfnorm{W_n} \\
&\les \sumN \chi^2_n(t) \tilde{\l}^{1+\frac{\b}{2}}_n,
\end{align*}
whence for any $n \geq N$,
\begin{align*}
\tl_n^{1+\frac{\b}{2}} \mu \big( t \in [0,T]: \twonorm{\divv(u\otimes u)} > \tl^{1+\frac{\b}{2}}_n \big) &\les \tl^{1+\frac{\b}{2}}_n |\tau_n -\tau_{n+1}| \\
&\les \tl^{1+\frac{\b}{2}}_n \tau_n \les 1.
\end{align*}
These establish \ref{lemma:forced_Euler_properties_4}.
\end{proof}
\begin{Proposition}\label{prop:Onsager_space_second_Euler}
The Euler solution $u$ defined in \eqref{def:second_construction_Euler} satisfies
\[
u \in L^3\big([0,T];B^{\frac{1}{3},b_n}_{3,\infty}(\T^3)\big)
\quad \text{for any } \{b_n\}_{n\in\mathbb{N}} \in \ell^1_+.
\]
Moreover,
\begin{equation}\label{lemma:Onsager_space_second_Euler_1}
\limsup_{q \to \infty}
\int_0^T
\tilde{\lambda}_q \|\Delta_q u(t)\|_{L^3}^3\,dt
\lesssim 1.
\end{equation}
\end{Proposition}

\begin{proof}
By
Lemma~\ref{lemma:build_block_second_construction},\ref{lemma:build_block_second_construction_2},
using that $\Delta_n W_n = W_n$ and $\operatorname{supp}\chi_n$ has size
$|\tau_n-\tau_{n+1}| \sim \tilde{\lambda}_n^{-1-\beta/2}$ yields
\begin{align*}
\int_0^T \|u(t)\|_{B^{\frac{1}{3},b_n}_{3,\infty}}^3\,dt
&\le
\sum_{n=N}^\infty
|\operatorname{supp}\chi_n|\,
b_n \,
\tilde{\lambda}_n \|\Delta_n W_n\|_{L^3}^3
\\
&\lesssim
\sum_{n=N}^\infty
\tilde{\lambda}_n^{-1-\frac{\beta}{2}}
\,b_n \,
\tilde{\lambda}_n \big(\tilde{\lambda}_n^{\beta(\frac{1}{2}-\frac{1}{3})}\big)^3
\\
&\lesssim
\sum_{n=N}^\infty b_n < \infty,
\end{align*}
which establishes the Besov membership. Next, using
Lemma~\ref{lemma:build_block_second_construction}\ref{lemma:build_block_second_construction_1},
the Littlewood–Paley blocks are localized:
\[
\Delta_q u(t) = \chi_q(t)\,W_q.
\]
Thus
\begin{align*}
\int_0^T
\tilde{\lambda}_q \|\Delta_q u(t)\|_{L^3}^3\,dt
&=
\int_0^T
\tilde{\lambda}_q\,\chi^3_q(t) \|W_q\|_{L^3}^3\,dt
\\
&\lesssim
\tilde{\lambda}_q \big(\tilde{\lambda}_q^{\beta(\frac{1}{2}-\frac{1}{3})}\big)^3
|\operatorname{supp}\chi_q|
\\
&\lesssim
\tilde{\lambda}_q^{1+\frac{\beta}{2}}
\tilde{\lambda}_q^{-1-\frac{\beta}{2}}
\lesssim 1.
\end{align*}
Taking $\limsup_{q\to\infty}$ of the above yields
\eqref{lemma:Onsager_space_second_Euler_1}.
\end{proof}
\begin{Proposition}\label{prop:zero_AW}
For the solution $u$ defined by \eqref{def:second_construction_Euler}, the energy
jump equals the amount of anomalous dissipation at time $T$, that is,
\begin{equation}\label{prop:zero_AW_1}
\Bigg|
\frac{1}{2}\|u(T)\|_{L^2}^2
-
\frac{1}{2}\lim_{t\to T^-}\|u(t)\|_{L^2}^2
\Bigg|
=
\frac{1}{2}
=
\Pi(T).
\end{equation}
Moreover, the anomalous work performed by $f$ vanishes:
\begin{equation}\label{prop:zero_AW_2}
\Phi(T)=0.
\end{equation}
\end{Proposition}

\begin{proof}
By lemma \ref{lemma:forced_Euler_properties} the solution $(u,f)$ satisfies the energy jump formula \eqref{eq:energy_jump_2}:
\begin{align*}
\haf \twonorm{u(T)}^2 - \haf  \twonorm{u(T-h)}^2 &= \lim_{t \to T^-}\int_{T-h}^{t} \<f,u \> d\tau + \limsup_{q \to \infty} \Big( -\Pi_q(T,h) + \Phi_q (T,h) \Big) \\
&= -\haf.
\end{align*}
Therefore it suffices to show that for any $h > 0$ small enough,
\begin{equation}\label{proof:prop_zero_AW_1} 
\lim_{q \to \infty} \Pi_q(T,h) = \haf.
\end{equation}
\noindent From the definition \eqref{def:time_gluing_second_construction} we see that $\chi_n \cap \chi_m = \emptyset$ when $|n-m| > 1$ and so by lemma \ref{lemma:build_block_second_construction}, \ref{lemma:build_block_second_construction_1}, for any $t \in (0,T)$ it holds that 
\[
\supp \Hat{u}(t) \subset \{k: \tilde{\lambda}_{n-1} \leq |k| \leq \tilde{\lambda}_{n+1} \}.
\]
Thus if $q > n$ then $\Ss_q u(t) = u(t)$, therefore $\Pi_q(T,h) = 0$ since $u(t)$ is incompressible. On the other hand if $q < n-1$, then $\Ss_q u(t) = 0$ and again $\Pi_q (T,h) = 0$. As a result we can write
\begin{equation}\label{proof:prop_zero_AW_2}
\int_{\T^3} \divv (u \otimes u) \Ss_q (\Ss_q u) dx = \chi^3_q (t) \int_{\T^3} \divv \Big[ \Ss_q \big( W_q \otimes W_q \big) \Big] \cdot (\Ss_q W_q) dx + \iota_q (t),  
\end{equation}
where 
\[
\iota (t) := \int_{\T^3} \divv \Ss_q \Big ( \sum_{j = q-1}^{q+1} \chi_j W_j \otimes \sum_{j = q-1}^{q+1} \chi_j W_j \Big) \cdot \Big( \Ss_q \sum_{j = q-1}^{q+1} \chi_j W_j \Big) dx - \chi^3_q (t) \int_{\T^3} \divv \Big[ \Ss_q \big( W_q \otimes W_q \big) \Big] \cdot (\Ss_q W_q) dx.
\]
Due to lemma \ref{lemma:build_block_second_construction}, \ref{lemma:build_block_second_construction_3}, when all indices in the sum are equal to $q-1$ or $q+1$, the first term becomes zero. Then by \ref{lemma:build_block_second_construction_2} in the same lemma one can obtain the following estimate:
\begin{equation}\label{proof:prop_zero_AW_3}
|\iota_q (t)| \les \tilde{\l}^{1+\frac{\b}{2}}_q \Big(\chi_{q-1} (t) + \chi_{q+1} (t) \Big) \chi^2_q(t).
\end{equation}
Integrating \eqref{proof:prop_zero_AW_2} over time $T-h$ to $T$, we arrive at
\begin{align}
\Pi_q(T,h) &= \int_{T-h}^T \int_{\T^3} \divv \Big[ \Ss_q (u \otimes u) \cdot (\Ss_q u) \Big] dx dt \notag \\ 
& = \int_{T-h}^T \chi^3_q (t) \int_{\T^3} \divv \Big[ \Ss_q \big( W_q \otimes W_q \big) \Big] \cdot (\Ss_q W_q) dx dt + \int_{T-h}^T \iota_q (t) dt. \label{proof:prop_zero_AW_4}
\end{align}
From the definition of $\chi_n(t)$ and $\tau_n$ one sees that 
\begin{align*}
&|\{t: \chi_q(t)= 1 \}| = \tau_q - \tau_{q+1} - c_\e \tau^{1+\frac{\e}{4}}_{q+1} = \frac{1}{4} \tl^{-1-\frac{\b}{2}}_q - c_\e \tau^{1+\frac{\e}{4}}_{q+1}, \\
&|\{t: 0< \chi_q (t) < 1 \}| \leq c_\e \big( \tau^{1+\frac{\e}{4}}_{q} + \tau^{1+\frac{\e}{4}}_{q+1}  \big) \leq 2 c_\e \tau^{1+\frac{\e}{4}}_{q},
\end{align*}
whence 
\begin{align*}
 &\lim_{q \to \infty} \int_{\{ \chi_q(t) = 1\} } \int_{\T^3} \divv \Big[ \Ss_q \big( W_q \otimes W_q \big) \Big] \cdot (\Ss_q W_q) dx dt \\
 =&\lim_{q \to \infty} \haf \bigg ( \haf \tilde{\l}^{-1-\frac{\b}{2}}_q \int_{\T^3} \divv \Big[ \Ss_q \big( W_q \otimes W_q \big) \Big] \cdot (\Ss_q W_q) dx - c_\e \tau^{1+\frac{\e}{4}}_{q+1} \int_{\T^3} \divv \Big[ \Ss_q \big( W_q \otimes W_q \big) \Big] \cdot (\Ss_q W_q) dx \bigg) \\
 = & \ \ \haf
\end{align*}
due to lemma \ref{lemma:build_block_second_construction}, \ref{lemma:build_block_second_construction_3}. Additionally,
\begin{align*}
&\lim_{q \to \infty} \Big| \int_{\{ 0< \chi_q(t) < 1\} } \int_{\T^3} \divv \Big[ \Ss_q \big( W_q \otimes W_q \big) \Big] \cdot (\Ss_q W_q) dx dt \Big| \\
\leq & 2c_\e \lim_{q \to \infty} \tau_q^{1+\frac{\e}{4}} \Big |  \int_{\T^3} \divv \Big[ \Ss_q \big( W_q \otimes W_q \big) \Big] \cdot (\Ss_q W_q) dx \Big| = 0.
\end{align*}
The above yields 
\begin{equation}\label{proof:prop_zero_AW_5}
\lim_{q \to \infty} \int_{T-h}^T \chi^3_q (t) \int_{\T^3} \divv \Big[ \Ss_q \big( W_q \otimes W_q \big) \Big] \cdot (\Ss_q W_q) dx dt = \haf.
\end{equation}
For the second term of \eqref{proof:prop_zero_AW_2}, we have 
\begin{equation}\label{proof:prop_zero_AW_6}
\lim_{q \to \infty} \Big| \int_{T-h}^T \iota_q(t) dt \Big| \les \lim_{q \to \infty} \tilde{\l}^{1+\frac{\b}{2}}_q\int_{\supp \chi_{q-1} \cap \supp \chi_q} dt  \les \lim_{q \to \infty}\tilde{\l}^{1+\frac{\b}{2}}_q \tau^{1+\frac{\e}{4}}_q = 0,
\end{equation}
in which we used \eqref{proof:prop_zero_AW_3}. Therefore from \eqref{proof:prop_zero_AW_5} and \eqref{proof:prop_zero_AW_6} we established \eqref{proof:prop_zero_AW_1}.
\end{proof}
\smallskip
\section{Proof of Theorem \ref{main_Theorem_3}}\label{sec:proof_main_Theorem_3}

\noindent For $m \in \N$, we define
\begin{equation}\label{def:u_m_second_construction}
u^m (t) := \sum_{k = N}^{N+m} \chi_k(t) W_k (x) 
\end{equation}
and 
\begin{equation}\label{def:f_nu_m_second_construction}
    f^{\nu_m} := \p_t u^m + \divv (u^m \otimes u^m) + \nu_m \Delta u^m
\end{equation}
where
\begin{equation}\label{def:nu_m_second_construction}
    \nu_m := \left ( 2 \int_0^T \twonorm{\nabla u^m(t)}^2 dt  \right )^{-1}.
\end{equation}
It follows from the lemma \ref{lemma:build_block_second_construction} that the size of $\nu_m$ is 
\[
\nu_m \sim \tilde{\l}^{-1+\frac{\b}{2}}_m, \ \ \b < 2
\]
and so we have $\nu_m \to 0$ as $m \to \infty$. Now, for all $t \in \left[ 0,T \right)$, we have 
\begin{align*}
    \twonorm{u^m(t) - u(t)}^2 = \sum_{k = N+m+1}^\infty \chi^2_k(t) \twonorm{W_k}^2 = \sum_{k=N+m+1}^\infty \chi^2_k(t) \to 0, \ \ \ \text{as} \ \ m \to \infty.
\end{align*}
For $t = T$, we have that $\twonorm{u(T)} = \twonorm{u^m(T)} = 0$ for all $m \in \N$, so the convergence follows as well. Hence, $u^m \to u$ in $L^\infty([0,T];L^2(\T^3))$. For the convergence $f^{\nu_m} \to f$ in $L^1([0,T];L^r(\T^3))$, $r < 2$, it suffices to show that, as $m \to \infty$, 
\begin{subequations}\label{eq:convergence_of_force_second_construction}
\begin{align}
\|\p_t u^m - \p_t u\|_{L^1_t L^r_x} &\to 0, \label{eq:convergence_of_force_second_construction_1} \\
\|\div(u^m \otimes u^m)-\div(u \otimes u)\|_{L^1_t L^r_x} &\to 0 ,\label{eq:convergence_of_force_second_construction_2} \\
\nu_m \|\Delta u^m \|_{L^1_t L^r_x} &\to 0 .\label{eq:convergence_of_force_second_construction_3}
\end{align}
\end{subequations}
To see \eqref{eq:convergence_of_force_second_construction_1}, we use lemma \ref{lemma:build_block_second_construction}, \ref{lemma:build_block_second_construction_2} to obtain
\begin{align*}
    \|\p_t u^m -\p_t u \|_{L^1_t L^r_x} &\les \sum_{n=N+1+m}^\infty \int_0^T |\chi'_n(t)|\|W_n\|_{L^r} dt \les \sum_{n = N+1+m}^\infty \tilde{\l}^{\b(\haf-\frac{1}{r})}_n \tau^{1+\frac{\e}{4}}_n \tau^{-(1+\frac{\e}{4})}_n, \\
    &\les \sum_{n = N+1+m}^\infty \tilde{\l}^{\b(\haf-\frac{1}{r})}_n \to 0, \ \ \text{as} \ m \to \infty,  
\end{align*}
for $r < 2$. The convergence \eqref{eq:convergence_of_force_second_construction_2} holds due to the following
\begin{align*}
    \rnorm{u^m \cdot \nabla u^m - u \cdot \nabla u} &\leq \rnorm{(u^m - u) \cdot \nabla u^m} + \rnorm{u \cdot \nabla (u^m-u)}, \\ 
    &\leq \|u^m - u\|_{L^2} \| \nabla u^m\|_{L^{\frac{2r}{2-r}}} + \|u\|_{L^2} \|\nabla u^m -\nabla u\|_{L^{\frac{2r}{2-r}}}.
\end{align*}
Note that for any $t \in [0,T]$
\[
\|u^m(t) - u(t) \|_{L^2(\T^3)} \les \sum_{n >m} |\chi_n(t)| \| W_n \|_{L^2} = \sum_{n >m} |\chi_n(t)| \to 0, \ \ \text{as} \ m \to \infty.
\]
by the definition of $\chi_n(t)$. Meanwhile, 
\begin{align*}
    \int_0^T \| \nabla u^m(t)\|_{L^{\frac{2r}{2-r}}} dt &\les \sum_{n=N}^m \int_0^T |\supp \chi_n| \| \nabla W_n \|_{L^{\frac{2r}{2-r}}} dt \les \sum_{n =N }^m \int_0^T |\tau_n| \tilde{\l}_n \tilde{\l}_n^{\b(\haf - \frac{2-r}{2r})} dt \\
    &\les \sum_{n = N}^m \tilde{\l}_n^{-(1+\frac{\b}{2})} \tilde{\l}_n \tilde{\l}_n^{\b(\haf-\frac{1}{r}+\haf)}\les \sum_{n = N}^m \tilde{\l}^{\b(\haf-\frac{1}{r})}_n,
\end{align*}
therefore we have 
\[
\sup_{m > N}  \int_0^T \| \nabla u^m(t)\|_{L^{\frac{2r}{2-r}}} dt < \infty.
\]
Similarly, 
\begin{align*}
    \int_0^T \|\nabla u^m -\nabla u\|_{L^{\frac{2r}{2-r}}} dt \les \sum_{n > m} \int_0^T |\tau_n| \|\nabla W_n \|_{L^{\frac{2r}{2-r}}} dt \les \sum_{n >m } \tilde{\l}^{\b(\haf - \frac{1}{r})}_n \to 0, \ \ \text{as} \ m \to \infty. 
\end{align*}
Hence, 
\begin{align*}
    \|u^m \cdot \nabla u^m - u \cdot \nabla u\|_{L^1_t L^r} &= \int_0^T \|u^m \cdot \nabla u^m - u \cdot \nabla u\|_{L^r_x}dt, \\
    &\leq \int_0^T \|(u^m -u) \cdot \nabla u^m \|_{L^r}dt + \int_0^T \|u \cdot \nabla (u^m - u)\|_{L^r}dt, \\
    &\leq  \|u^m -u \|_{L^\infty_t L^2}  \bigg( \sup_{m > N}  \int_0^T \| \nabla u^m(t)\|_{L^{\frac{2r}{2-r}}} dt\bigg), \\
    &+ \|u\|_{L^\infty_t L^2} \int_0^T \|\nabla u^m -\nabla u \|_{L^{\frac{2r}{2-r}}} dt \to 0, \ \ \text{as} \ m \to \infty.
\end{align*}
Recalling from \eqref{def:nu_m_second_construction} that $\nu_m$ is of the size $\nu \sim \tilde{\l}^{-1+\frac{\b}{2}}_m$, the convergence \eqref{eq:convergence_of_force_second_construction_3} follows as before, thanks to the lemma \ref{lemma:build_block_second_construction}:
\begin{align*}
    \nu_m \int_0^T \|\Delta u^m(t) \|_{L^r} dt &\les \nu_m \sum_{n=N}^m \int_0^T |\supp \chi_n(t)| \tilde{\l}^2_n \|W_n\|_{L^r} dt, \\
    &\les \nu_m \sum_{n =N}^m |\tau_n| \tilde{\l}^2_n \tilde{\l}^{\b(\haf-\frac{1}{r})}_n \les \tilde{\l}^{-1+\frac{\b}{2}}_m \sum_{n=N}^m \tilde{\l}^{-1-\frac{\b}{2}+2+\frac{\b}{2}-\frac{\b}{r}}_n \les \tilde{\l}^{-1+\frac{\b}{2}+1-\frac{\b}{r}}_m \to 0,
\end{align*}
as $m \to \infty$. These prove Theorem \ref{main_Theorem_3},  \eqref{main_Theorem_3_1}. In addition, \eqref{main_Theorem_3_2} and \eqref{main_Theorem_3_3} are implied by proposition \ref{prop:Onsager_space_second_Euler} and \ref{prop:zero_AW}. Finally, since 
\[
\twonorm{u_\initial}^2 = \twonorm{u^m(0)}^2 = \sum_{n=N}^{N+m} \chi^2_n(0) \twonorm{W_n}^2 = \sum_{n=N}^{N+m} \chi^2_n(0) = 1
\]
and by \eqref{def:nu_m_second_construction} 
\[
2 \nu_m \int_0^T \twonorm{\nabla u^m(t)}^2 dt = 1, \ \ \text{for all} \ m > N,
\]
we have that 
\[
\lim_{m \to \infty} 2 \nu_m \int_0^T \twonorm{\nabla u^m(t)}^2dt = 1 = \twonorm{u_\initial}^2.
\]
The zero limiting work is then followed by the energy equality for the NSE solutions and convergence of $u^m$ in $L^\infty(0,T;L^2(\T^3))$, which concludes the proof of the Theorem \ref{main_Theorem_3}. \qed

\bibliographystyle{plain}
\bibliography{DA_With_Exp_Timescale}

\end{document}